\documentclass{amsart}
\usepackage{amsmath}
\usepackage{graphicx}
\usepackage{latexsym}
\usepackage{color}
\newtheorem{thm}{Theorem}[section]
\newtheorem{lem}[thm]{Lemma}
\newtheorem{cor}[thm]{Corollary}
\newtheorem{prop}[thm]{Proposition}
\newtheorem{rem}[thm]{Remark}

\newtheorem{ques}[thm]{Question}

\newtheorem{defn}[thm]{Definition}

\newcommand{\CP}{{\mathbb{CP}}{}^{2}}

\newcommand{\Z}{\mathbb{Z}}

\newcommand{\Mod}{{\rm Mod}}
\newcommand{\red}{\textcolor{red}}

\newcommand{\blue}{\textcolor{blue}}

\begin{document}

\title[Lefschetz pencils and finitely presented groups]
{Lefschetz pencils and finitely presented groups}

\author[R. Kobayashi]{Ryoma Kobayashi}
\address{Department of Mathematics, Faculty of Science and Technology, Tokyo University of Science, Noda, Chiba, 278-8510, Japan}
\email{kobayashi\_ryoma@ma.noda.tus.ac.jp}
\author[N. Monden]{Naoyuki Monden}
\address{Department of Engineering Science, Osaka Electro-Communication University, Hatsu-cho 18-8, Neyagawa, 572-8530, Japan}
\email{monden@isc.osakac.ac.jp}

\begin{abstract}
In this paper, given a finitely presented group $\Gamma$, we provide the explicit monodromy of a Lefschetz fibration with $(-1)$-sections whose total space has fundamental group $\Gamma$ by applying ``twisted substitutions" 
to that of the Lefschetz fibration constructed by Cadavid and independently Korkmaz. 
Consequently, we obtain an upper bound for the minimum $g$ such that there exists a genus-$g$ Lefschetz pencil on a smooth 4-manifold whose fundamental group is isomorphic to $\Gamma$. 
\end{abstract}

\maketitle

\setcounter{secnumdepth}{2}
\setcounter{section}{0}
\section{Introduction}
From the remarkable works of \cite{Do} and \cite{GS}, a $4$-manifold admits a symplectic structure if and only if it admits a genus-$g$ 
Lefschetz pencil for some $g$. 
This fact naturally raises the following basic question which remains open: 
\begin{ques}[cf. \cite{KS}]\label{Pencil1} 
Given a symplectic 4-manifold, what is the minimal genus $g$ for which it has a genus-$g$ Lefschetz pencil ? 
\end{ques}
Gompf \cite{Gompf} showed that every finitely presented group is realized as the fundamental group of some closed symplectic 4-manifold. 
It immediately follows that for any finitely presented group $\Gamma$, there exists a genus-$g$ Lefschetz pencil on a symplectic 4-manifold with fundamental group $\Gamma$ for some $g$. 
Therefore, it is natural to ask the following weaker version of Question~\ref{Pencil1}: 
\begin{ques}\label{Pencil2} 
Given a finitely presentaed group $\Gamma$, what is the minimal genus, denoted by $g_P(\Gamma)$, for which it has a genus-$g$ Lefschetz pencil on a symplectic 4-manifold with fundamental group $\Gamma$? 
\end{ques}
Our purpose is to give an upper bound for $g_P(\Gamma)$. 
To state our main result, we need to introduce some notation.

\begin{defn}\label{syllable}\rm
Let $\Gamma=\langle x_1,x_2,\ldots,x_n\mid r_1,r_2\ldots,r_k \rangle$ be a finitely presented group with $n$ generators and $k$ relations. 
For $w\in \Gamma$, we define $l(w)$, called the \textit{syllable length} of $w$, to be 
\begin{align*}
l(w)=\min\{s\mid w=x_{i_1}^{m_1}x_{i_2}^{m_2}\cdots x_{i_s}^{m_s}, \ 1\leq i_j\leq n, \ m_j\in\mathbb{Z}\}.
\end{align*}
Define $l=\max\{l(r_i)\mid 1\leq i\leq k\}$. 
If $k=0$, we define $l=1$. 
\end{defn}
Note that $l$ depends on the presentation and that our definition of $l$ differs from that of \cite{Kor2}. 
We always assume that the relators $r_i$ are cyclically reduced. 
Our main result is the following: 
\begin{cor}\label{pencil}
Let $\Gamma$ be a finitely presented group with a presentation in Definition~\ref{syllable}. 
Then, for $g\geq 4(n+l-1)+k$, there exists a genus-$g$ Lefschetz pencil on a closed symplectic 4-manifold $X$ such that $\pi_1(X)$ is isomorphic to $\Gamma$. 
Therefore, $g_P(\Gamma)\leq 4(n+l-1)+k$. 
\end{cor}

By blowing up the base locus of a genus-$g$ Lefschetz pencil, we obtain a genus-$g$ Lefschetz fibration over $S^2$, especially with $(-1)$-sections. 
Therefore, the above Gompf's result can be expressed in terms of Lefschetz fibrations. 
That is, for any finitely presented group $\Gamma$, there exists a Lefschetz fibration over $S^2$ (with $(-1)$-sections) such that the fundamental group of the total space is $\Gamma$. 
Note that, for $g\geq 2$, the total space of a genus-$g$ Lefschetz fibration 
is symplectic (cf. \cite{GS}).

The article \cite{ABKP} gave another construction of a Lefschetz fibration over $S^2$ whose total space has the given fundamental group. 
In the construction, the genus and the monodromy of the Lefschetz fibration are implicit. 
The explicit monodromies of such Lefschetz fibrations were given by Korkmaz \cite{Kor2} by using twisted fiber sum operations. 
Akhemdov and Ozbagci \cite{AO} gave new construction of such Lefschetz fibrations 
by performing Luttinger surgeries and knot surgeries on the symplectic sum of certain symplectic $4$-manifolds. 
The first author \cite{Kob} improved the result of \cite{Kor2} by using twisted fiber sum operations.

The examples in \cite{Kor2}, \cite{AO} and \cite{Kob} have no $(-1)$-sections. 
Another purpose of this paper is to provide explicit monodromies of Lefschetz fibrations admitting a $(-1)$-section with a prescribed fundamental group. 
We prove the following: 
\begin{thm}\label{main}
Let $\Gamma$ be a finitely presented group with a presentation in Definition~\ref{syllable}. 
Then, for $g\geq 4(n+l-1)+k$, there exists a genus-$g$ Lefschetz fibration $X\rightarrow S^2$ with two disjoint $(-1)$-sections such that $\pi_1(X)$ is isomorphic to $\Gamma$.
\end{thm}
We give the monodoromies of our fibrations by using ``twisted substitution" technique introduced in this paper and \cite{HKM}. 
In order to use the technique, we construct a certain relation in the mapping class group of a surface with two boundary components. 
This gives a lift of the monodromy of Gurtas's fibration (see \cite{Gur}) to the mapping class group of a surface with two boundary components. 
By blowing down one of the disjoint $(-1)$-sections of our fibrations, we obtain Corollary~\ref{pencil}.

We would like to emphasize that fiber sum operations and symplectic sum operations can not be effectively utilized to construct Lefschetz fibrations with $(-1)$-sections. 
In fact, it is well-known that the fiber sum of Lefschetz fibrations has no $(-1)$-sections from the work of \cite{St3} (see also \cite{Sm4}). 
More strongly, Usher \cite{Us} (and see also \cite{Sato2}) showed that the symplectic sum of symplectic $4$-manifolds is minimal, that is, it does not contain any $(-1)$-spheres. 
Moreover, Luttinger surgery and knot surgery preserve minimality of symplectic 4-manifolds from Usher's result. 
For this reason, the fibrations in \cite{Kor2}, \cite{AO} and \cite{Kob} have no $(-1)$-sections. 
We note that, in general, it is difficult to compute the genus of a Lefschetz pencil on the total space of a given Lefschetz fibration without a $(-1)$-section.


Here is an outline of this paper. 
In Section~\ref{notation}, we fix the notations. 
In Section~\ref{mapping}, we introduce a twisted substitution technique for constructing a new word in mapping class groups and the relator in mapping class groups constructed by Korkmaz. 
Section~\ref{Lefschetz} reviews some standard facts on Lefschetz fibrations and pencils. 
In Section~\ref{proofmain}, we prove the main results. 
In the last section, we give an alternative construction of the monodromy of Gurtas' fibration and provide a lift of that to the mapping class group of a surface with two boundary components. 
In Appendix A, we introduce the construction of a loop which is needed for the proof of Theorem~\ref{main}.

\noindent \textit{Acknowledgments.} 
The authors would like to thank Susumu Hirose for his comments on this paper. 
The second author was supported by Grant-in-Aid for Young Scientists (B) (No. 13276356), Japan Society for the Promotion of Science.


\section{Notation}\label{notation}

\

Let $\Sigma_g$ be the closed oriented surface of genus $g$ standardly embedded in the 3-space as shown in Figure~\ref{fundamental}. 
We will use the symbols $a_1,b_1,\ldots,a_g,b_g$ to denote the standard generators of the fundamental group $\pi_1(\Sigma_g)$ of $\Sigma_g$ as shown in Figure~\ref{fundamental}. 
For $a$ and $b$ in $\pi_1(\Sigma_g)$, the notation $ab$ means that we first apply $a$ then $b$. 
\begin{figure}[hbt]
 \centering
     \includegraphics[width=9cm]{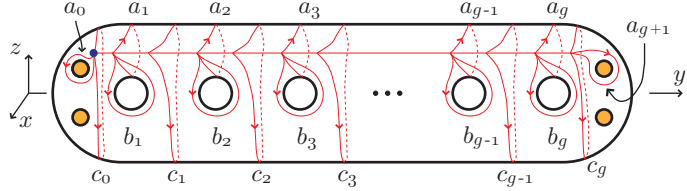}
     \caption{Generators $a_j,b_j$ of the fundamental group and loops $c_j$.}
     \label{fundamental}
\end{figure}

Let $c_0,c_1,c_2\ldots,c_g,a_0,a_{g+1}$ be the simple loops in $\Sigma_g$ as shown in Figure~\ref{fundamental}, and note that in $\pi_1(\Sigma_g)$, up to conjugation, 
\begin{align}
&c_i=b_i^{-1}\cdots b_1^{-1}(a_1b_1a_1^{-1})\cdots (a_ib_ia_i^{-1}); \label{c_i}\\
&c_0=c_g=1;\label{c_g}\\
&a_0=a_{g+1}=1\label{a_{g+1}}
\end{align}
for each $1\leq i\leq g$. 
Then, the fundamental group $\pi_1(\Sigma_g)$ has the following presentation:
\begin{align*}
\pi_1(\Sigma_g)=\langle a_1,b_1,\ldots, a_g,b_g \mid c_g\rangle.  
\end{align*}

Let $B_0,B_1,B_2,\ldots, B_g$, $a_1^\prime,\ldots,a_g^\prime$ be the simple closed curves in $\Sigma_g$ as shown in Figure~\ref{CK}. 
Suppose that $g=2r$. 
Then, it is easy to check that up to conjugation, the following equalities hold in $\pi_1(\Sigma_g)$: 
\begin{align}
&B_{2k-1}=a_kb_kb_{k+1} \cdots b_{g+1-k}c_{g+1-k}a_{g+1-k} &&\mathrm{for} \ \ 1\leq k\leq r; \label{B2k-1}\\
&B_{2k}=a_kb_{k+1}b_{k+2} \cdots b_{g-k}c_{g-k}a_{g+1-k} &&\mathrm{for} \ \ 0\leq k\leq r; \label{B2k}\\
&a_{k+1}^\prime=c_ka_{k+1} &&\mathrm{for} \ 0\leq k\leq g-1, \label{aprime}
\end{align}
If $g=2r+1$, then $B_{2k-1}$ satisfies the equality (\ref{B2k-1}) for $1\leq k\leq r+1$.

Let $A_1\ldots,A_{2g+1}$ be the simple closed curves on $\Sigma_g$ as shown in Figure~\ref{gurtas1curve}. 
It is easily seen that, up to conjugation, the following equalities hold in $\pi_1(\Sigma_g)$:
\begin{align}
&A_{2k}=b_k &&\mathrm{for} \ \ 1\leq k\leq g; \label{E2k} \\
&A_{2k+1}=a_ka_{k+1}^{-1} &&\mathrm{for} \ \ 0\leq k\leq g. \label{E2k-1}
\end{align}

\begin{figure}[hbt]
 \centering
      \includegraphics[width=8cm]{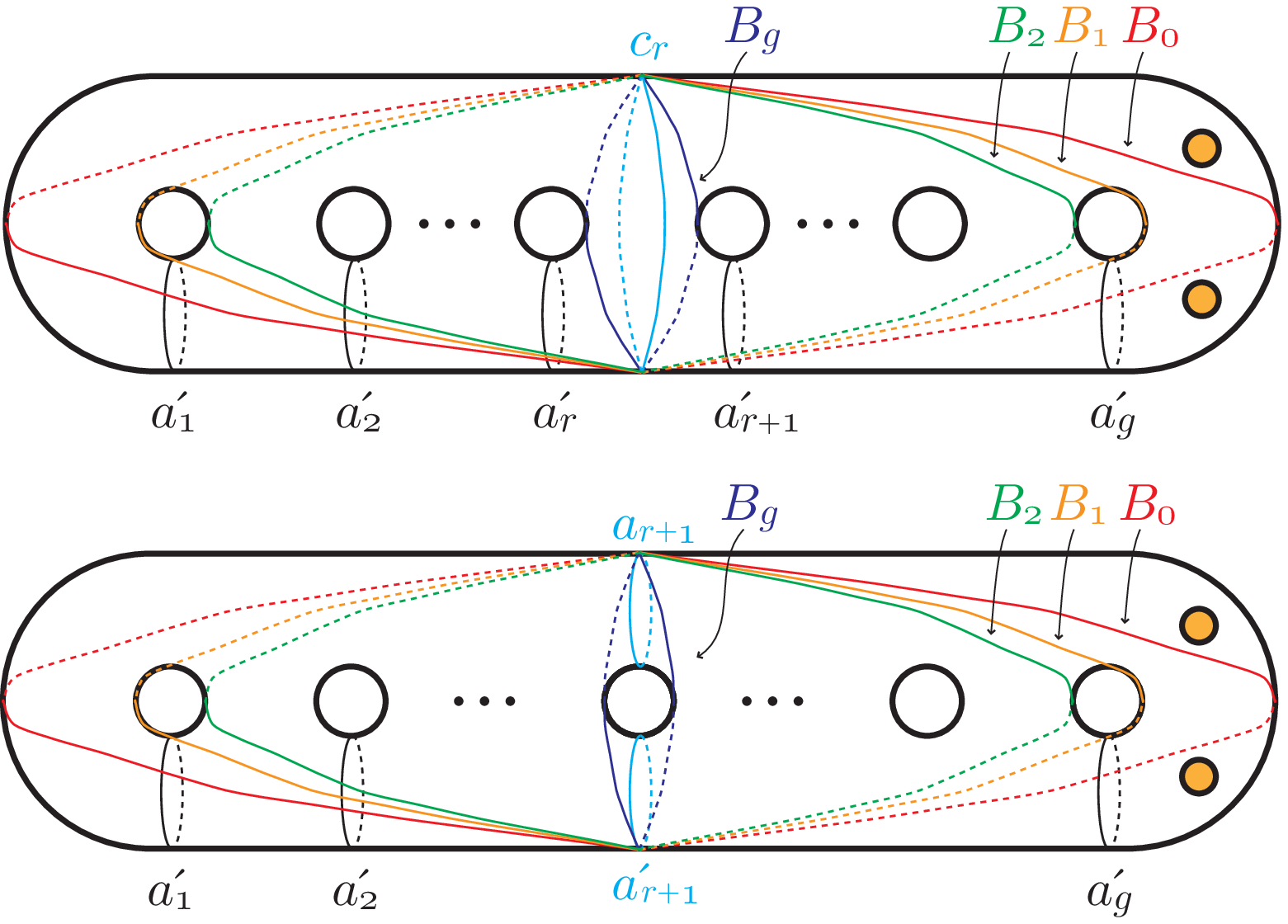}
      \caption{The curves $B_0, B_1, B_2, \ldots, B_g, a_1^\prime, \ldots, a_g^\prime$.}
      \label{CK}
 \end{figure}

Moreover, when we denote by $D_0,D_1,D_2,\ldots,D_{2h_1}$ and $E_{h_1}$ the simple closed curves on $\Sigma_g$ as shown in Figure~\ref{gurtas1curve}, 
it is immediate that, up to conjugation, the following equalities hold in $\pi_1(\Sigma_g)$:
\begin{align}
&D_0=b_1b_2\cdots b_{2h_1}a_{2h_1+1}^{-1}; \label{D0}\\
&D_{2k-1}=a_kb_kb_{k+1} \cdots b_{2h_1+1-k}c_{2h_1+1-k}a_{2h_1+1-k}a_{2h_1+1}^{-1} &&\mathrm{for} \ \ 1\leq k\leq h_1; \label{D2k-1}\\
&D_{2k}=a_kb_{k+1}b_{k+2} \cdots b_{2h_1-k}c_{2h_1-k}a_{2h_1+1-k}a_{2h_1+1}^{-1} &&\mathrm{for} \ \ 1\leq k\leq h_1; \label{D2k}\\
&E_{h_1}=c_{h_1}a_{2h_1+1}. \label{E}
\end{align}
Note that we can modify $\Sigma_g$ and $D_0,D_1,D_2,\ldots,D_{2h_1}$ and $E_{h_1}$ by isotopy as shown in Figure~\ref{gurtas2curve}. 
\begin{figure}[hbt]
 \centering
      \includegraphics[width=12cm]{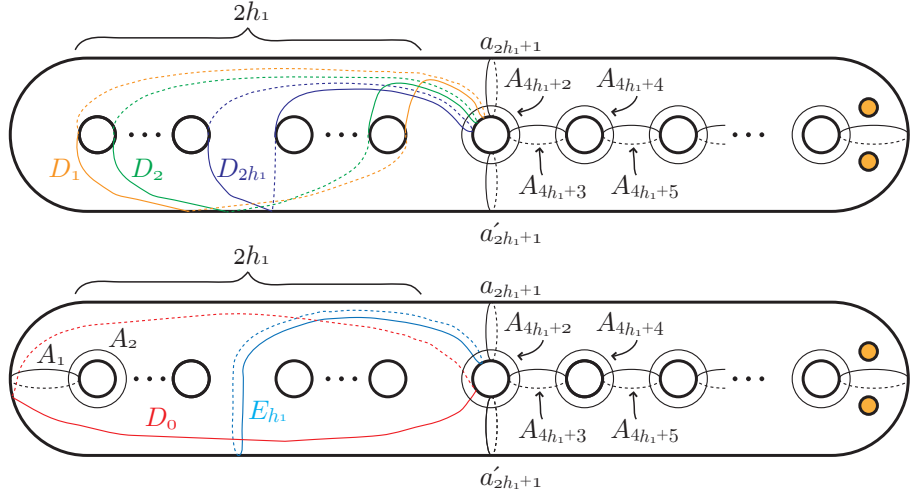}
      \caption{The curves $A_1,A_2,\ldots,A_{2g+1}$, $D_0, D_1, \ldots, D_{2h_1}$ and $E_{h_1}$.}
      \label{gurtas1curve}
 \end{figure}
\begin{figure}[hbt]
 \centering
      \includegraphics[width=8cm]{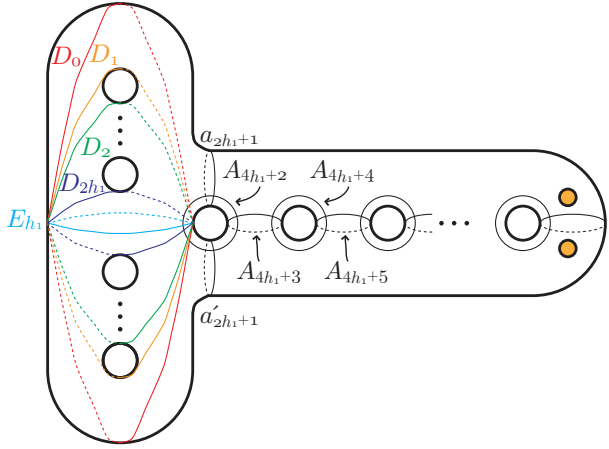}
      \caption{Modified surface $\Sigma_g$ and modified curves $D_0, D_1, \ldots, D_{2h_1}$ and $E_{h_1}$.}
      \label{gurtas2curve}
 \end{figure}

Throughout this paper, we use the same symbol for a loop and its homotopy class. 
Similarly, we use the same symbol for a diffeomorphism and its isotopy class, or a simple closed curve and its isotopy class. 
A simple loop and a simple closed curve will even be denoted by the same symbol. 
It will cause no confusion as it will be clear from the context which one we mean.

\section{Mapping class groups}\label{mapping}
\subsection{Twisted substitutions}\label{MCG}

\

Let $\Sigma_g^b$ be a compact oriented surface of genus $g$ with $b$ boundary components. 
The \textit{mapping class group} of $\Sigma_g^b$, denoted by $\mathrm{Mod}_g^b$, is 
the group of isotopy classes of orientation preserving self-diffeomorphisms of $\Sigma_g^b$. 
We assume that diffeomorphisms and isotopies fix the points of the boundary. 
To simplify notation, we write $\Sigma_g = \Sigma_g^0$ and $\mathrm{Mod}_g = \mathrm{Mod}_g^0$. 
For $\phi_1$ and $\phi_2$ in $\mathrm{Mod}_g^b$, the notation $\phi_1\phi_2$ means that we first apply $\phi_2$ then $\phi_1$ (Our notation differs from that of \cite{Kor2}).

For a simple closed curve $c$ on $\Sigma_g^b$, by cutting $\Sigma_g^b$ along $c$ and gluing the two boundary curves back after twisting of 
the sides to the right by $360^{\circ}$, we obtain a diffeomorphism $\Sigma_g^b\to\Sigma_g^b$. 
The isotopy class of such a diffeomorphism is called the \textit{right-handed Dehn twist} along $c$, denoted by $t_c$. 
Note that $t_{\phi(c)}=\phi t_c\phi^{-1}$ for an element $\phi$ in $\Mod_g^b$ and $t_ct_d=t_dt_c$ if $c$ is disjoint from $d$.

\begin{defn}\rm
A word $\varrho : = t_{c_1}^{\epsilon_1} t_{c_2}^{\epsilon_2} \cdots t_{c_n}^{\epsilon_n}$ in $\mathrm{Mod}_g^b$ is a \textit{relator} 
if $\varrho$ satisfies $\varrho=1$, where $\epsilon_i=\pm 1$. 
If $\epsilon_1=\cdots=\epsilon_n=1$, then $\varrho$ is called the \textit{positive relator}. 
\end{defn}

We introduce a main technique to construct new product of right-handed Dehn twists in $\mathrm{Mod}_g^b$ from old ones, called a \textit{twisted substitution}. 
This technique is a generalization of a substitution technique introduced by Fuller and is also introduced in \cite{HKM}.

\begin{defn}\rm
Let $\eta$ be the following relator in $\mathrm{Mod}_g^b$:
\begin{align*}
\eta : = \ t_{c_1} t_{c_2} \cdots t_{c_k} \cdot t_{d_l}^{-1} \cdots t_{d_2}^{-1} t_{d_1}^{-1} \ (=1). 
\end{align*}
For an element $\phi$ in $\mathrm{Mod}_g^b$ satisfying $\phi(d_i)=d_i$, by the relation $t_{\phi(c)}=\phi t_c\phi^{-1}$, 
we obtain the following relator in $\mathrm{Mod}_g^b$, denoted by $\eta^\phi$:
\begin{align*}
\eta^\phi : = \ t_{\phi(c_1)} t_{\phi(c_2)} \cdots t_{\phi(c_k)} \cdot t_{d_l}^{-1} \cdots t_{d_2}^{-1} t_{d_1}^{-1} \ (=1). 
\end{align*}
Let us denote by $\varrho$ a product of right-handed Dehn twists which includes $t_{d_1} \cdots t_{d_l}$ as a subword: 
\begin{align*}
\varrho = U \cdot t_{d_1} t_{d_2} \cdots t_{d_l} \cdot V, 
\end{align*}
where $U$ and $V$ are products of right handed Dehn twists. 
Then, we get a new product $\varrho^\prime$ of right handed Dehn twists 
\begin{align*}
\varrho^\prime := U \cdot t_{\phi(c_1)} t_{\phi(c_2)} \cdots t_{\phi(c_k)} \cdot V \ ( = U \cdot \eta^\phi \cdot t_{d_1} t_{d_2} \cdots t_{d_l} \cdot V). 
\end{align*}
Then, $\varrho^\prime$ is said to be obtained by applying a \textit{$\phi$-twisted $\eta$-substitution} to $\varrho$. 
We call an $\mathrm{id}$-twisted $\eta$-substitution a \textit{trivial} $\eta$-substitution. 
\end{defn}

\subsection{The relator $W_2^g$}\label{MCK-relator}

\

In this section, we introduce a relator $W_2^g$ in $\mathrm{Mod}_g^2$ that was introduced by Korkmaz \cite{Kor2}.

We denote by $\Sigma_g^2$ the surface of genus $g$ with two boundary components obtained from $\Sigma_g$ by removing two disjoint open disks (cf. Figure~\ref{CK}). 
Let $a_{g+1}$ be one of the boundary curves of $\Sigma_g^2$ as shown in Figure~\ref{fundamental}, and let $a_{g+1}^\prime$ be the other boundary curve defined by $a_{g+1}^\prime=c_ga_{g+1}$. 
Korkmaz gave the following relator $W_2^g$ in $\mathrm{Mod}_g^2$: 
\begin{align*}
&W_2^g:=
  \left\{ \begin{array}{ll}
      \displaystyle ( t_{B_0} t_{B_1} t_{B_2} \cdots t_{B_g} t_{c_r} )^2 t_{a_{g+1}}^{-1} t_{a^\prime_{g+1}}^{-1} & \ \ (g=2r) \\[3mm]
      \displaystyle ( t_{B_0} t_{B_1} t_{B_2} \cdots t_{B_g} t_{a_{r+1}}^2 t_{a_{r+1}^\prime}^2 )^2 t_{a_{g+1}}^{-1} t_{a_{g+1}^\prime}^{-1} & \ \ (g=2r+1).
      \end{array} \right.
\end{align*}

Since the simple closed curves $a_{g+1}$ and $a_{g+1}^\prime$ are null-homotopic in $\Sigma_g$, $t_{a_{g+1}}=t_{a_{g+1}^\prime}=1$ in $\mathrm{Mod}_g$. 
Therefore, in $\mathrm{Mod}_g$, the word $W_2^g$ is a positive relator. 
This positive relator was discovered by Matsumoto \cite{Ma} in $\mathrm{Mod}_2$ and a generalization of that constructed independently by Cadavid \cite{Ca} and Korkmaz \cite{Kor}.

Although in \cite{Kor2} Korkmaz does not prove the words $W_2^g$ to be relator in $\mathrm{Mod}_g^2$, 
we can prove it by applying the same argument in Section 2 of \cite{Kor}. 
In Section~\ref{Gurtas}, we give a very short outline of the proof (see Lemma~\ref{korkmaz}).

\section{Lefschetz pencils and fibrations}\label{Lefschetz}
\subsection{Basics on Lefschetz pencils and fibrations}

\

We recall the definition and basic properties of Lefschetz pencils and fibrations. 
More details can be found in \cite{GS}.

\begin{defn}\rm
Let $X$ be a closed, connected, oriented smooth $4$-manifold, 
and let $B=\{b_1,\ldots,b_m\}$ and $C=\{p_1,\ldots,p_n\}$ be finite, disjoint subsets of $X$. 

Let $f:X\setminus B\to S^2$ be a smooth map satisfies the following three conditions:
\begin{itemize}
\item[(a)] For each point $b_i\in B$, there are orientation-preserving complex coordinate charts on which $f$ is of the form $f(z_1, z_2) = z_1/z_2$, 
\item[(b)] $C$ is the set of critical points of $f$, and for each $p_i$ and $f(p_i)$, there are complex local coordinate charts agreeing with 
the orientations of $X$ and $S^2$ on which $f$ is of the form $f(z_{1},z_{2})=z_{1}z_{2}$, 
\item[(c)] For $q\in S^2-f(C)$, $f^{-1}(q)\cup B \subset X$ is diffeomorphic to $\Sigma_g$. 
\end{itemize}
Then, $f$ is called a genus-$g$ \textit{Lefschetz pencil} if $B$ is the non-empty set, and 
$f$ is called a genus-$g$ \textit{Lefschetz fibration} if $B$ is the empty set. 
\end{defn}

The set $B$ is called the \textit{base locus}, and for each $q\in S^2$, $f(q)^{-1}\cup B$ is called the \textit{fiber} of $f$. 
We assume that $f$ is injective on $C$ and that $f$ is relatively minimal (i.e. no fiber contains a sphere with self-intersection number $-1$). 
A fiber containing a critical point is called a \textit{singular fiber}. 
Each singular fiber is obtained by collapsing a simple closed curve, called the \textit{vanishing cycle}, in the regular fiber to a point.

Once we fix an identification of $\Sigma_{g}$ with the fiber over a base point of $S^2-f(C)$, 
we can characterize the Lefschetz fibration $f:X\rightarrow S^2$ by its \textit{monodromy representation} 
$\pi_1(S^2-f(C))\rightarrow \mathrm{Mod}_{g}$. 
Note that in this paper, this map is an anti-homomorphism. 
Let $\gamma_1,\ldots,\gamma_n$ be an ordered system of generating loops for $\pi_1(S^2-f(C))$, 
such that each $\gamma_i$ encircles only $f(p_i)$ and $\gamma_1\gamma_2\cdots\gamma_n$ is homotopically trivial. 
Thus, since the monodromy of the fibration along each of the loops $\gamma_i$ is a right-handed Dehn twist along the corresponding vanishing cycle, 
the monodromy of $f$ comprises a positive relator 
\begin{eqnarray*}
t_{v_n} \cdots t_{v_2} t_{v_1} = 1 \in \mathrm{Mod}_g, 
\end{eqnarray*}
where $v_i$ are the corresponding vanishing cycles of the singular fibers. 
Conversely, for any positive relator $\varrho\in\mathrm{Mod}_g$, we can construct a genus-$g$ Lefschetz fibration over $S^2$ whose monodromy is $\varrho$. 
Therefore, we denote a genus-$g$ Lefschetz fibration associated to a positive relator $\varrho$ in $\mathrm{Mod}_g$ by $f_\varrho:X_\varrho \rightarrow S^2$.

According to theorems of Kas \cite{Kas} and Matsumoto \cite{Ma}, 
if $g\geq 2$, then the isomorphism class of a Lefschetz fibration is determined by a positive relator modulo \textit{simultaneous conjugations} 
\begin{align*}
t_{v_n}\cdots t_{v_2}t_{v_1} \sim t_{\phi(v_n)}\cdots t_{\phi(v_2)}t_{\phi(v_1)} \ \ {\rm for \ any} \ \phi \in \Gamma_g
\end{align*}
and \textit{elementary transformations} 
\begin{align*}
&t_{v_n}\cdots t_{v_{i+2}}t_{v_{i+1}}t_{v_i}t_{v_{i-1}}t_{v_{i-2}}\cdots t_{v_1}& &\sim& &t_{v_n}\cdots t_{v_{i+2}}t_{v_i}t_{t_{v_i}^{-1}(v_{i+1})}t_{v_{i-1}}t_{v_{i-2}}\cdots t_{v_1},&\\
&t_{v_n}\cdots t_{v_{i+2}}t_{v_{i+1}}t_{v_i}t_{v_{i-1}}t_{v_{i-2}}\cdots t_{v_1}& &\sim& &t_{v_n}\cdots t_{v_{i+2}}t_{v_{i+1}}t_{t_{v_i}(v_{i-1})}t_{v_i}t_{v_{i-2}}\cdots t_{v_1}.&
\end{align*}

\subsection{Sections of Lefschetz fibrations}\label{sections}

\

\begin{defn}\rm
For a Lefschetz fibration $f:X\rightarrow S^2$, a map $\sigma:S^2\rightarrow X$ is called a \textit{$k$-section} of $f$ 
if $f\circ \sigma={\rm id}_{S^2}$ and the self-intersection number of the homology class $[\sigma(S^2)]$ in $H_2(X;\Z)$ is equal to $k$. 
\end{defn}

Let $d_1,d_2,\ldots,d_b$ be $b$ boundary curves of $\Sigma_g^b$. 
Then, a lift of a positive relator $\varrho = t_{v_n} \cdots t_{v_2} t_{v_1} = 1$ in $\mathrm{Mod}_g$ to $\mathrm{Mod}_g^b$ as 
\begin{align*}
t_{v_n^\prime} \cdots t_{v_2^\prime} t_{v_1^\prime} \cdot  t_{d_b}^{-n_b} \cdots t_{d_2}^{-n_2} t_{d_1}^{-n_1} = 1 
\end{align*}
shows that
the existence of $b$ disjoint sections $\sigma_1,\ldots,\sigma_b$ of $f_\varrho$ such that for each $1\leq i\leq b$, the self-intersection of $\sigma_i$ is equal to $-n_i$. 
Here, $v_i^\prime$ is a simple closed curve mapped to $v_i$ under $\Sigma_g^b \to \Sigma_g$. 
Conversely, if a genus-$g$ Lefschetz fibration admits disjoint $b$ sections $\sigma_1,\ldots,\sigma_b$ such that for each $1\leq i\leq b$, the self-intersection of $\sigma_i$ is equal to $-n_i$, 
then we obtain such a relator in $\mathrm{Mod}_g^b$.

From the definitions of Lefschetz fibrations and pencils, the blow-up all points of $B=\{q_1,\ldots,q_m\}$ of a genus-$g$ Lefschetz pencil 
yields a genus-$g$ Lefschetz fibration with $m$ disjoint $(-1)$-sections. 
From the above fact, this gives a relation 
\begin{align*}
t_{v_n^\prime} \cdots t_{v_2^\prime} t_{v_1^\prime} \cdot  t_{d_m}^{-1} \cdots t_{d_2}^{-1} t_{d_1}^{-1} = 1 
\end{align*}
in $\mathrm{Mod}_g^m$. 
Conversely, such a relator determines a genus-$g$ Lefschetz fibration with $m$ disjoint $(-1)$-sections, and 
by blowing these sections down, we can obtain a genus-$g$ Lefschetz pencil.

When a Lefschetz fibration $X\to S^2$ admit a section, we can compute the fundamental group of $X$ as follows.

\begin{lem}[cf.\cite{GS}]\label{lem2}\rm
Let $\varrho$ be a positive relator $\varrho = t_{v_n} \cdots t_{v_2} t_{v_1}$ of $\mathrm{Mod}_g$. 
Suppose that a genus-$g$ Lefschetz fibration $f:X_\varrho\to S^2$ admits a section $\sigma$. 
Then, the fundamental group $\pi_1(X)$ is isomorphic to the quotient of $\pi_1(\Sigma_g)$ by the normal subgroup generated by $v_1,\ldots,v_n$. 
\end{lem}

\section{Proof of Theorem~\ref{main}}\label{proofmain}

For a finitely presented group
$\Gamma=\langle{x_1,x_2,\dots,x_n\mid{r_1,r_2,\dots,r_k}}\rangle$ with
$n$ generators and $k$ relators, let
$l=\max\{l(r_i)\mid1\leq{i}\leq{k}\}$, where $l(r_i)$ is the syllable
length of $r_i$.
In this section, we suppose $h_1\geq n+l-1$ and $2(h_2-1)\geq k$.

\subsection{Construction of a word $W^g(1,\phi)$}

\

In this subsection, we construct key relator in $\Mod_g^2$. 
Let us consider $\Sigma_g^2$ obtained from $\Sigma_g$ by removing two disjoint open disks (see Figure~\ref{fundamental}, \ref{CK} and~\ref{gurtas1curve}). 
Write $r=2h_1+h_2-1$.

\begin{prop}\label{gurtasrelation}
In $\Mod_g^2$, the following relations hold, where $g\geq r$.
\begin{align*}
V_1:&=t_{c_r}^{-1} \cdot t_{E_{h_1}} t_{A_{4h_1+2}} \cdots t_{A_{2r}} t_{a_r} t_{a_r^\prime} t_{A_{2r}} \cdots t_{A_{4h_1+2}} t_{E_{h_1}} \\
& \ \cdot t_{a_r^\prime} t_{A_{2r}} \cdots t_{A_{4h_1+2}} \cdot (t_{D_0} t_{D_1} \cdots t_{D_{2h_1}})^2 \cdot t_{A_{4h_1+2}} \cdots t_{A_{2r}} t_{a_r^\prime}, \\[3pt]
V_2:&= t_{a_{r+1}^\prime}^{-1}t_{a_{r+1}}^{-1} \cdot t_{E_{h_1}} t_{A_{4h_1+2}} \cdots t_{A_{2r}} t_{a_r} t_{a_r^\prime} t_{A_{2r}} \cdots t_{A_{4h_1+2}} t_{E_{h_1}} \\
& \ \cdot t_{A_{2r+1}} \cdots t_{A_{4h_1+2}} \cdot (t_{D_0} t_{D_1} \cdots t_{D_{2h_1}})^2 \cdot t_{A_{4h_1+2}} \cdots t_{A_{2r+1}}. 
\end{align*}
\end{prop}
We postpone the proof of Proposition~\ref{gurtasrelation} until Section~\ref{Gurtas} (see Proposition~\ref{gurtasx}).

Suppose that $g=2r$, 
Then, we can find two $t_{c_r}$ in the relator $W_2^g$. 
Therefore, by Proposition~\ref{gurtasrelation}, we can apply once trivial $V_1$-substitution and once $\phi$-twisted $V_1$-substitution to $W_2^g$ for some mapping class $\phi$ in $\Mod_g^2$ satisfying $\phi(c_r)=c_r$.

Suppose that $g=2r+1$.
Since then $t_{a_{r+1}}^2t_{a_{r+1}^\prime}^2=(t_{a_{r+1}}t_{a_{r+1}^\prime})^2$, 
we can find four $t_{a_{r+1}}t_{t_{a_{r+1}^\prime}}$ in the relator $W_2^g$. 
Therefore, by Proposition~\ref{gurtasrelation}, we can apply once trivial $V_2$-substitution and once $\phi$-twisted $V_2$-substitution to $W_2^g$ for some mapping class $\phi$ in $\Mod_g^2$ satisfying $\phi(a_{r+1})=a_{r+1}$ and $\phi(a_{r+1}^\prime)=a_{r+1}^\prime$. 
Note that in this case, we may choose any two $t_{a_{r+1}}t_{a_{r+1}^\prime}$ from among these and apply only twice (twisted) $V_2$-substitutions to $W_2^g$.

Then, we denote by 
\begin{align*}
W_2^g(1,\phi)
\end{align*}
the relator in $\Mod_g^2$ which is obtained by applying once trivial $V_1$-substitution and once $\phi$-twisted $V_1$-substitution to $W_2^g$ if $g=2r$, 
and once trivial $V_2$-substitution and once $\phi$-twisted $V_2$-substitution to $W_2^g$ if $g=2r+1$.

Since $t_{a_{g+1}}=1$ and $t_{a_{g+1}^\prime}=1$ in $\Mod_g$, the relator $W_2^g(1,\phi)$ in $\Mod_g$ is a positive relator. 
Therefore, we obtain a genus-$g$ Lefschetz fibration $f_{W_2^g(1,\phi)}$ with two disjoint $(-1)$-sections. 
The vanishing cycles of this fibration are 
\begin{align*}
&B_0, B_1, \ldots, B_g, \\
&D_0, D_1\ldots, D_{2h_1}, E_{h_1}, A_{4h_1+2}, \ldots, A_{2r}, a_r, a_r^\prime, \\
&\phi(D_0), \phi(D_1), \ldots, \phi(D_{2h_1}), \phi(E_2), \phi(A_{4h_1+2}), \ldots, \phi(A_{2r}), \phi(a_r), \phi(a_r^\prime) 
\end{align*}
if $g=2r$, and 
\begin{align*}
&B_0, B_1, \ldots, B_g, a_{r+1}, a_{r+1}^\prime, \\
&D_0, D_1, \ldots, D_{2h_1}, E_{h_1}, A_{4h_1+2}, \ldots, A_{2r+1}, a_r, a_r^\prime, \\
&\phi(D_0), \phi(D_1), \ldots, \phi(D_{2h_1}), \phi(E_{h_1}), \phi(A_{4h_1+2}),\ldots, \phi(A_{2r+1}), \phi(a_r), \phi(a_r^\prime) 
\end{align*}
if $g=2r+1$.

\subsection{Proof for free groups}\label{freegroup}

\

We will construct Lefschetz fibrations with two disjoint $(-1)$-sections whose fundamental groups are free groups in Section~\ref{freegroup}. 
In order to prove this result, we prepare the following Lemma.

\begin{lem}\label{pi1}
Let $r=2h_1+h_2-1$. 
Let $\langle S \rangle$ be the normal closure of the elements of the set $S$ of the following simple closed curves on $\Sigma_g$:
\begin{align*}
&S=\{B_0, B_1, \ldots, B_g, D_0, D_1, \ldots, D_{2h_1}, E_{h_1}, A_{4h_1+2}, \ldots, A_{2r}, a_r, a_r^\prime \}
\end{align*}
if $g=2r$, and 
\begin{align*}
&S= \{B_0, B_1, \ldots, B_g, a_{r+1}, a_{r+1}^\prime, D_0, D_1, \ldots, D_{2h_1}, E_{h_1}, A_{4h_1+2}, \ldots, A_{2r+1}, a_r, a_r^\prime \}
\end{align*}
if $g=2r+1$. 
Then, $\pi_1(\Sigma_g)/\langle S \rangle$ has a presentation with generators $a_1,b_1,\ldots, a_g, b_g$ and with relations 
\begin{align*}
&a_ia_{g+1-i}=b_ia_{g+1-i}b_{g+1-i}a_{g+1-i}^{-1}=1 &&\mathrm{for} \ \ 1\leq i\leq r; \\
&a_{2h_1+k}=b_{2h_1+k}=1 &&\mathrm{for} \ \ 1\leq k\leq h_2-1; \\
&a_ja_{2h_1+1-j}=b_ja_{2h_1+1-j}b_{2h_1+1-j}a_{2h_1+1-j}^{-1}=1 &&\mathrm{for} \ \ 1\leq j\leq h_1; \\ 
&c_{h_1}=1. &&
\end{align*}
if $g=2r$, and 
\begin{align*}
&a_ia_{g+1-i}=b_ia_{g+1-i}b_{g+1-i}a_{g+1-i}^{-1}=1 &&\mathrm{for} \ \ 1\leq i\leq r; \\
&a_{2h_1+k}=b_{2h_1+k}=1 &&\mathrm{for} \ \ 1\leq k\leq h_2-1; \\
&a_ja_{2h_1+1-j}=b_ja_{2h_1+1-j}b_{2h_1+1-j}a_{2h_1+1-j}^{-1}=1 &&\mathrm{for} \ \ 1\leq j\leq h_1; \\
&a_{r+1}=c_{h_1}=1. &&
\end{align*}
if $g=2r+1$. 
\end{lem}
\begin{proof}
Suppose that $g=2r$. 
From the equalities (\ref{B2k-1}) and (\ref{B2k}) in Section~\ref{notation}, in $\pi_1(\Sigma_g)/\langle S \rangle$, we have 
\begin{align}\label{arelation}
a_ia_{g+1-i}=1. 
\end{align}
This gives 
\begin{align*}
&1=B_{2i-1}=b_ib_{i+1}\cdots b_{g+1-i}c_{g+1-i} &&\mathrm{for} \ \ 1\leq i\leq r; \\
&1=B_{2i}=b_{i+1}b_{i+2}\cdots b_{g-i}c_{g-i} &&\mathrm{for} \ \ 1\leq i\leq r
\end{align*}
in $\pi_1(\Sigma_g)/\langle S \rangle$. 
From these two equalities, we have $b_ic_{g-i}^{-1}b_{g+1-i}c_{g+1-i}=1$ for each $1\leq i\leq r$ and 
\begin{align}\label{crelation}
&c_r=1.
\end{align}
Note that $c_{g+1-i}=b_{g+1-i}^{-1}c_{g-i}(a_{g+1-i}b_{g+1-i}a_{g+1-i}^{-1})$ from the equality (\ref{c_i}). 
Therefore, by $b_ic_{g-i}^{-1}b_{g+1-i}c_{g+1-i}=1$, we obtain 
\begin{align}\label{brelation}
b_ka_{g+1-i}b_{g+1-i}a_{g+1-i}^{-1}=1. 
\end{align}

From $a_r=1$, $A_l=1$ for $4h_1+2 \leq l\leq 2r$ and the equalities (\ref{E2k}) and (\ref{E2k-1}), we obtain 
\begin{align}\label{ab}
a_{2h_1+k}=b_{2h_1+k}=1
\end{align}
for $1\leq k\leq h_2-1$. 
From $a_r^\prime=1$, the equality (\ref{aprime}), (\ref{crelation}), (\ref{c_i}) and (\ref{ab}), we have 
\begin{align}\label{c2relation}
c_{r-1}=c_{2h_1}=1. 
\end{align}
By $a_{2h_1+1}=1$, $c_{2h_1}=1$ and the equalities (\ref{D0}), (\ref{D2k-1}) and (\ref{D2k}), 
a similar argument to the proofs of the relations (\ref{arelation}) and (\ref{brelation}) gives 
\begin{align}\label{abrelation}
&a_ja_{2h_1+1-j}=b_ja_{2h_1+1-j}b_{2h_1+1-j}a_{2h_1+1-j}^{-1}=1& &\mathrm{and}& &c_{h_1}=1&
\end{align}
for $1\leq j\leq 2h_1$.

From the equalities (\ref{arelation}), (\ref{crelation}), (\ref{brelation}), (\ref{ab}), (\ref{crelation}) and (\ref{abrelation}), we see that $\pi_1(\Sigma_g)/\langle S \rangle$ has a presentation with generators $a_1,b_1,\ldots, a_g, b_g$ and with relations 
\begin{align*}
&a_ia_{g+1-i}=b_ia_{g+1-i}b_{g+1-i}a_{g+1-i}^{-1}=1 &&\mathrm{for} \ \ 1\leq i\leq r; \\
&a_{2h_1+k}=b_{2h_1+k}=1 &&\mathrm{for} \ \ 1\leq k\leq h_2-1; \\
&a_ja_{2h_1+1-j}=b_ja_{2h_1+1-j}b_{2h_1+1-j}a_{2h_1+1-j}^{-1}=1 &&\mathrm{for} \ \ 1\leq j\leq h_1; \\ 
&c_g=c_r=c_{r-1}=c_{2h_1}=c_{h_1}=1. &&
\end{align*}
Then, by the equalities (\ref{c_i}), (\ref{ab}) and (\ref{abrelation}), we can delete the relation $c_g=c_r=c_{r-1}=c_{2h_1}=1$. 
This is our claim. 

Suppose that $g=2r+1$. 
Since $a_{r+1}=a_{r+1}^\prime=1$ and $a_{r+1}^\prime=c_ra_{r+1}$, we have $c_r=1$. 
A similar argument as in the case of $g=2r$ shows that $\pi_1(\Sigma_g)/\langle S \rangle$ has the presentation. 
This completes the proof. 
\end{proof}

We construct Lefschetz fibrations with two disjoint $(-1)$-sections whose fundamental groups are free groups.

Let $h_1\geq 1$ and $h_2-1\geq 1$. 
We define an element $\varphi$ in $\Mod_g^2$ to be 
\begin{align*}
\varphi=t_{a_{n+1}} t_{a_{n+2}} \cdots t_{a_{h_1}} t_{b_{h_1+1}} t_{b_{h_1+2}} \cdots t_{b_{2h_1}}. 
\end{align*}
Then, we see that $\varphi(c_{2h_1+h_2-1})=c_{2h_1+h_2-1}$ if $g=2(2h_1+h_2-1)$, 
and $\varphi(a_{2h_1+h_2})=a_{2h_1+h_2}$ and $\varphi(a_{2h_1+h_2}^\prime)=a_{2h_1+h_2}^\prime$ if $g=2(2h_1+h_2-1)+1$. 
Therefore, we can define the relator $W_2^g(1,\varphi)$ in $\Mod_g^2$.

\begin{prop}\label{free}
Let $\varphi=t_{a_{n+1}} t_{a_{n+2}} \cdots t_{a_{h_1}} t_{b_{h_1+1}} t_{b_{h_1+2}} \cdots t_{b_{2h_1}}$ in $\Mod_g$. 
If $g\geq 2(2n+1)$, 
then we have 
\begin{align*}
\pi_1(X_{W_2^g(1,\varphi)})\cong F_n, 
\end{align*}
where $F_n$ is a free group of rank $n$. 
\end{prop}
\begin{proof}
Let $h_1\geq n$ and $h_2-1\geq 1$. 
For simplicity of notation, we write $G$ instead of $\pi_1(X_{W_2^g(1,\varphi)})$. 

Suppose that $g=2(2h_1+h_2-1)$ and let $r=2h_1+h_2-1$. 
Note that $G$ has a presentation 
with generators $a_1, b_1,\ldots, a_g, b_g$ and with relations 
\begin{align*}
&c_g=1; && \\
&B_i=1 &&\mathrm{for} \ \ 0\leq i\leq g; \\ 
&a_r=a_r^\prime=E_{h_1}=1; && \\
&D_j=A_k=1 &&\mathrm{for} \ \ 0\leq j\leq 2h_1, \ \ 4h_1+2\leq k\leq 4h_1+2h_2-2; \\
&\varphi(a_r)=\varphi(a_r^\prime)=\varphi(E_{h_1})=1; && \\
&\varphi(D_j)=\varphi(A_k)=1 &&\mathrm{for} \ \ 0\leq j\leq 2h_1, \ \ 4h_1+2\leq k\leq 4h_1+2h_2-2. 
\end{align*}

It is easily seen that we have the following equalities (up to conjugation) in $\pi_1(\Sigma_g)$: 
\begin{align*}
&\varphi(D_0)=a_{h_1} \cdots a_{n+2}a_{n+1}D_0; \\
&\varphi(D_{2l-1})=b_{2h_1-l+1}^{-1}a_{h_1} \cdots a_{n+2}a_{n+1}D_{2l-1} &&\mathrm{for} \ \ 1\leq l\leq n; \\
&\varphi(D_{2l})=b_{2h_1-l+1}^{-1}a_{h_1} \cdots a_{n+2}a_{n+1}D_{2l} &&\mathrm{for} \ \ 1\leq l\leq n.
\end{align*}
Thus, by $D_0=\varphi(D_0)=D_j=\varphi(D_j)=1$ for $1\leq j\leq 2h_1$, we obtain 
\begin{align*}
b_{2h_1-l+1}=1 \ \ \ \mathrm{for} \ \ \ 1\leq l\leq n. 
\end{align*}

Similarly, we have the following equalities (up to conjugation) in $\pi_1(\Sigma_g)$: 
\begin{align*}
&\varphi(D_{2l-1})=b_{2h_1-l+1}^{-1}a_{h_1} \cdots a_{l+1}a_lD_{2l-1} &&\mathrm{for} \ \  n+1\leq l\leq r-1; \\
&\varphi(D_{2l})=b_{2h_1-l+1}^{-1}a_{h_1} \cdots a_{l+2}a_{l+1}D_{2l-1} &&\mathrm{for} \ \  n+1\leq l\leq r-1; \\
&\varphi(D_{2h_1-1})=b_{h_1+1}^{-1}a_{h_1}D_{2h_1-1}; \\
&\varphi(D_{2h_1})=b_{h_1+1}^{-1}B_{2h_1}. 
\end{align*}
By $D_j=1$ for $1\leq j\leq 2h_1$, $\varphi(D_{2l-1})=\varphi(D_{2l})=1$ for $n+1\leq l\leq h_1$, we obtain 
\begin{align*}
a_l=1 \ \ \ \mathrm{for} \ \ \ n+1\leq l\leq h_1. 
\end{align*}
Moreover, by $\varphi(D_{2l})=\varphi(D_{2l+1})=\varphi(D_{2h_1})=1$ for $n+1\leq l\leq h_1-1$, we have 
\begin{align*}
b_{2h_1-l+1}=1 \ \ \ \mathrm{for} \ \ \ n+1\leq l\leq h_1. 
\end{align*}

Here, since $\varphi(a_r)=a_r$, $\varphi(a_r^\prime)=a_r^\prime$, $\varphi(E_{h_1})=E_{h_1}$ and $\varphi(A_k)=A_k$ in $\pi_1(\Sigma_g)$ for each $4h_1+2\leq k\leq 4h_1+2h_2-2$, 
we can delete the relations $\varphi(a_r)=1, \varphi(a_r^\prime)=1$, $\varphi(E_{h_1})=1$ and $\varphi(A_K)=1$ from the above presentation of $G$.

From the above arguments and Lemma~\ref{pi1}, we see that $G$ has a presentation with generators $a_1,b_1,\ldots,a_g,b_g$ and with relations 
\begin{align*}
&a_ia_{g+1-i}=b_ia_{g+1-i}b_{g+1-i}a_{g+1-i}^{-1} &&\mathrm{for} \ \ 1\leq i\leq r; \\
&a_{2h_1+k}=b_{2h_1+k}=1 &&\mathrm{for} \ \ 1\leq k\leq h_2-1; \\
&a_ja_{2h_1+1-j}=b_ja_{2h_1+1-j}b_{2h_1+1-j}a_{2h_1+1-j}^{-1}=1 &&\mathrm{for} \ \ 1\leq j\leq h_1; \\
&c_{h_1}=1; &&\\
&a_{n+1}=a_{n+2}=\cdots =a_{h_1}=1;\\ 
&b_{h_1}=b_{h_1+1}=\cdots =b_{2h_1}=1.
\end{align*}
It is easily shown that this is a presentation of he free group of rank $n$, with a free basis $a_1\ldots,a_n$, 
that is, $G$ is isomorphic to $F_n$. 

The proof for $g=2r+1$ is similar. 
This completes the proof of Proposition~\ref{free}. 
\end{proof}


\subsection{Construction of Loops}\label{loop}

\

In order to prove Theorem~\ref{main}, we will use the following result. 
\begin{prop}\label{prop5.1}
Let $F_n$ be the subgroup of $\pi_1(\Sigma_n)$ generated by the generators $a_1,\ldots, a_n$, i.e., $F_n$ is a free group of rank $n$. 
Let $r_1,\ldots,r_k$ be arbitrary $k$ elements in $F_n$ represented as words in $a_1,\ldots,a_n$. 
Let $l=\max_{1\leq i\leq k}\{l(r_i)\}$, where $l(r_i)$ is the syllable lengths of $r_i$. 
Then, there are loops $R_1,\ldots,R_k$ on $\Sigma_{n+l-1}$ with the following property:
For each $1\leq i\leq k$,
\begin{enumerate}
\item[(a)] $R_i$ is a simple loop on $\Sigma_{n+l-1}$. 
\item[(b)] $R_i$ is freely homotopic to a simple closed curve which intersects $a_{n+l-1}$ transversely at only one point. 
\item[(c)] $\Phi([R_i])=r_i$, where $[R_i]\in\pi_1(\Sigma_{n+l-1})$ is the homotopy class of $R_i$, and $\Phi:\pi_1(\Sigma_{n+l-1})\to \pi_1(\Sigma_n)$ is the map defined by $\Phi(a_j)=a_j$ for $1\leq j\leq n$ and $\Phi(\alpha)=1$ 
for $\alpha\in\{a_{n+1},\ldots, a_{n+l-1},b_1,\ldots,b_{n+l-1}\}$. 
\end{enumerate}
\end{prop}

Actually, Proposition~\ref{prop5.1} was essentially proved by Korkmaz (Proposition 4.3 in \cite{Kor2}). 
In Proposition 4.3 of \cite{Kor2}, he defined $l$ as $l=l(r_1)+\cdots+l(r_k)$. 
However, we find that it is sufficient to consider $l$ as $l=\max_{1\leq i\leq k}\{l(r_i)\}$. 
We introduce the proof of Proposition~\ref{prop5.1} in Appendix~\ref{A}.

Let $h_1\geq n+l-1$ and $2(h_2-1)\geq k$, and let $g=2(2h_1+h_2-1)$ or $g=2(2h_1+h_2-1)+1$. 
Let us consider a surface $\Sigma_{n+l-1}$ and the loops $R_1,\ldots, R_k$ constructed in Proposition~\ref{prop5.1}. 
We remove a small open disk from $\Sigma_{n+l-1}$ near $a_{n+l-1}$ and disjoint from all $R_i$ (cf. Figure~\ref{Ri} (a)). 
Let us denote by $\Sigma_{n+l-1}^1$ the resulting surface of genus $n+l-1$ with one boundary component. 
We embed $\Sigma_{n+l-1}^1$ into the standard surface $\Sigma_g^2$ in such a way that 
for each $1\leq t\leq n+l-1$, simple loops $a_t,b_t$ on $\Sigma_{n+l-1}^1$ correspond to the simple loops $a_t,b_t$ on $\Sigma_g^2$ (cf. Figure~\ref{Ri} (b)). 
Then, we can modify $R_1,\ldots,R_k$ so that $R_1$ intersects $a_{2h_1+h_2-1}$ at one point and does not intersect $A_{4h_1+2},\ldots,A_{4h_1+2h_2-2}$, 
and $R_i$ intersects $A_{4h_1+2h_2-i}$ at one point for each $i=2,\ldots,2h_2-2$ and does not intersect $a_{2h_1+h_2-1}$ and $A_j$ for any $i\neq j$, where $j=2,\ldots,2h_2-2$. 
For example, we replace $R_i$ with a simple representative of $[R_i](b_{2h_1+1}b_{2h_2+2}\cdots b_{2h_1+h_2-i})^{\epsilon}$ if $i$ is odd, and $[R_i]a_{2h_1+h_2-i}^{\epsilon}$ if $i$ is even, where $\epsilon=\pm 1$ (cf. Figure~\ref{Ri} (c)). 
\begin{figure}[hbt]
 \centering
     \includegraphics[width=12.5cm]{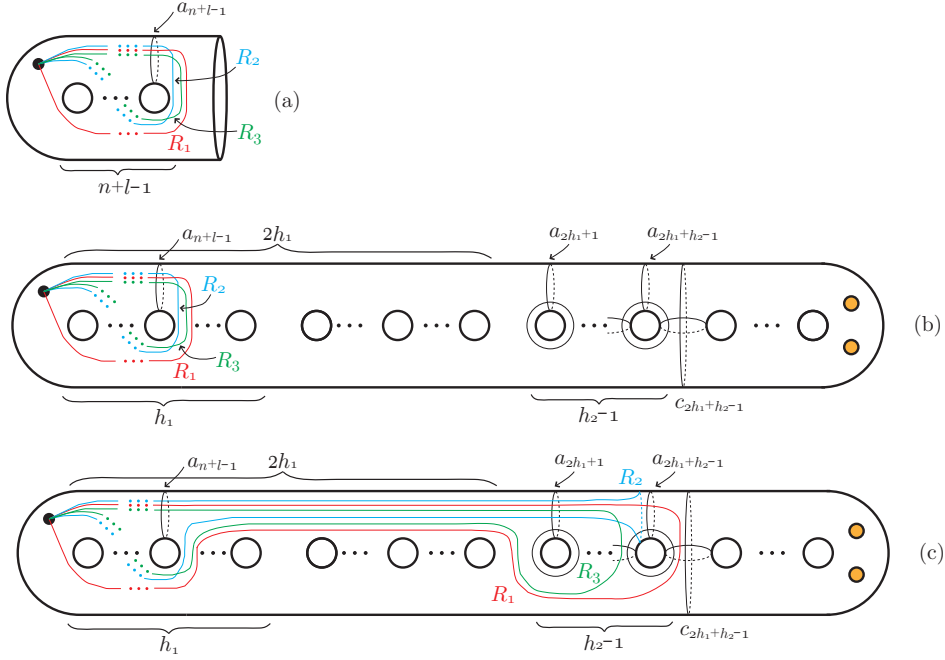}
     \caption{Modified curves $R_1,\ldots,R_k$ on $\Sigma_{g}$ for $g=2(2h_1+h_2-1)$.}
     \label{Ri}
\end{figure}

\subsection{Proof for arbitrary finitely presented group}

\

In this section, we prove Theorem~\ref{main}. 

We define an element $\psi$ in $\Mod_g^2$ to be 
\begin{align*}
\psi_1=t_{R_1} t_{R_2} \cdots t_{R_k} t_{a_{n+1}} t_{a_{n+2}} \cdots t_{a_{h_1}} t_{b_{h_1+1}} t_{b_{h_1+2}} \cdots t_{b_{2h_1}}, 
\end{align*}
where each $R_i$ is the loop in $\Sigma_g^2$ obtained by the construction in Section~\ref{loop}. 
Then, we see that $\psi_1(c_{2h_1+h_2-1})=c_{2h_1+h_2-1}$ if $g=2(2h_1+h_2-1)$, 
and $\psi_1(a_{2h_1+h_2})=a_{2h_1+h_2}$ and $\psi_1(a_{2h_1+h_2}^\prime)=a_{2h_1+h_2}^\prime$ if $g=2(2h_1+h_2-1)+1$. 
Therefore, we can define the relator $W_2^g(1,\psi_1)$ in $\Mod_g^2$. 
Note that in $\Mod_g$, $W_2^g(1,\psi_1)$ is a positive relator.

We now prove Theorem~\ref{main}. 
The proof is inspired by \cite{Kor2} and that of Proposition 13 in \cite{AO}. 

\begin{proof}[Proof of Theorem~\ref{main}]
Let $h_1\geq n+l-1$ and $2(h_2-1)\geq k$.
We will show that 
\begin{align*}
\pi_1(X_{W_2^g(1,\psi_1)})\cong \Gamma,
\end{align*}
where $\psi_1=t_{R_1} t_{R_2} \cdots t_{R_k} t_{a_{n+1}} t_{a_{n+2}} \cdots t_{a_{h_1}} t_{b_{h_1+1}} t_{b_{h_1+2}} \cdots t_{b_{2h_1}}$ in $\Mod_g$ 
and each $R_i$ is the loop in $\Sigma_g$ obtained by the construction in Section~\ref{loop}. 
For simplicity, we write $G^\prime$ instead of $\pi_1(X_{W_2^g(1,\psi_1)})$.

Suppose that $g=2(2h_1+h_2-1)$. 
Let 
\begin{align*}
&\varphi=t_{a_{n+1}} t_{a_{n+2}} \cdots t_{a_{h_1}} t_{b_{h_1+1}} t_{b_{h_1+2}} \cdots t_{b_{2h_1}}, \ \ \ \mathrm{and}\\
&\psi_i=t_{R_i} \cdots t_{R_k} \varphi, 
\end{align*}
and let $V$ be the set of the vanishing cycles of $f_{W_2^g(1,\psi)}$.

Since $R_1$ intersects $a_{2h_1+h_2-1}$ at one point and does not intersect $A_j$ for $j=4h_1+2, \ldots,A_{4h_1+2h_2-2}$, and $a_{2h_1+h_2-1}$ are disjoint from $a_{n+1},\cdots,a_{h_1},b_{h_1+1},\ldots,b_{2h_1}$ and $R_2,\ldots,R_k$, 
we see that in $\pi_1(\Sigma_g)$, up to conjugation, 
\begin{align*}
\psi_1(a_{2h_1+h_2-1})=t_{R_1}(a_{2h_1+h_2-1})=a_{2h_1+h_2-1}R_1^{\epsilon}, 
\end{align*}
where $\epsilon_1$ is equal to $1$ or $-1$. 
Since $a_{2h_1+h_2-1}=1$ in $G$, we may replace the relator $\psi_1(a_{2h_1+h_2-1})=1$ by $R_1=1$. 
Let $c$ be an element of $V$. 
If $R_1$ is disjoint from $\psi_2(c)$, then we have $\phi(c)=t_{R_1}(\psi_2(c))=\psi_2(c)$. 
If $R_1$ intersects $\psi_2(c)$ at $t$ points. 
then it is easily seen that there are elements $x_1,\ldots,x_{t+1}$ in $\pi_1(\Sigma_g)$ such that 
$\psi_2(c)=x_1x_2\cdots x_{t+1}$ and that $t_{R_1}(\psi_2(c))=x_1R_1^{\zeta_1}x_2R_1^{\zeta_2}\cdots x_tR_1^{\zeta_t}x_{t+1}$ (up to conjugacy), 
where each $\zeta_s$ is equal to $1$ or $-1$. 
From $R_1=1$, we obtain $\psi(c)=t_{R_1}(\psi_2(c))=\psi_2(c)$ in $G^\prime$. 
Therefore, we may replace the relator $\psi_1(c)=1$ by $\psi_2(c)=1$.

By repeating this argument for each $i=2,\ldots,k$,  we see that we may replace the relators $\psi_1(A_{4h_1+2h_2-i})=1$ and $\psi_1(c)=1$ by $R_i=1$ and $\varphi(c)=1$, respectively. 
In particular, since for each $j=4h_1+2,\ldots,4h_1+2h_2-2$, $a_{2h_1+h_2-1}=1$ and $A_j=1$ in $G^\prime$ and $a_{2h_1+h_2-1}=\varphi(a_{2h_1+h_2-1})$ and $A_j=\varphi(A_j)$ in $\pi_1(\Sigma_g)$ (up to conjugation), 
we can delete the relator $\psi_1(a_{2h_1+h_2-1})=1$ and $\psi_1(A_j)=1$. 
Therefore, from the proof of Proposition~\ref{free}, we see that $G^\prime$ has a presentation with generators $a_1, b_1,\ldots, a_g, b_g$ and with relations 
\begin{align*}
&a_ia_{g+1-i}=b_ia_{g+1-i}b_{g+1-i}a_{g+1-i}^{-1} &&\mathrm{for} \ \ 1\leq i\leq r; \\
&a_{2h_1+k}=b_{2h_1+k}=1 &&\mathrm{for} \ \ 1\leq k\leq h_2-1; \\
&a_ja_{2h_1+1-j}=b_ja_{2h_1+1-j}b_{2h_1+1-j}a_{2h_1+1-j}^{-1}=1 &&\mathrm{for} \ \ 1\leq j\leq h_1; \\
&c_{h_1}=1; &&\\
&a_{n+1}=a_{n+2}=\cdots =a_{h_1}=1;\\ 
&b_{h_1}=b_{h_1+1}=\cdots =b_{2h_1}=1;\\
&R_1=R_2=\cdots =R_k=1.
\end{align*}

We note that the element $[R_i]\in\pi_1(\Sigma_g)$ is contained in the subgroup generated by $a_1,b_1, \ldots,a_{h_1},b_{h_1}$ and $a_{2h_1+1},b_{2h_1+1},\ldots,a_{2h_1+h_2-1},b_{2h_1+h_2-1}$. 
Since from this presentation, we see that $a_s=1$ for $s=n+1,\ldots,h_1,2h_1+1,\ldots,2h_1+h_2-1$ and $b_j=1$ for $j=1,\ldots,h_1,2h_1+1,\ldots,2h_1+h_2-1$, we get a word representing the element $r_i$ by Proposition~\ref{prop5.1}. 
Therefore, $G$ is isomorphic to $\Gamma$.

A similar argument works for $g=2(2h_1+h_2-1)+1$. 
This completes the proof of Theorem~\ref{main}.
\end{proof}

\begin{proof}[Proof of Corollary~\ref{pencil}]
The genus-$g$ Lefschetz fibration in Theorem~\ref{main} has at least two disjoint $(-1)$-sections. 
Therefore, by blowing down one of the disjoint $(-1)$-sections of this Lefschetz fibration, we obtain the required genus-$g$ Lefschetz pencil. 
This completes the proof. 
\end{proof}

\begin{rem}
The upper bound for $g_P(\Gamma)$ in Corollary~\ref{pencil} may not be sharp. 
In fact, since $\CP$ admits a genus-$0$ Lefschetz pencil, $g_P(\Gamma)=0$ if $\Gamma$ is the trivial group. 
When we replace the relations in Proposition~\ref{gurtasrelation} and the map $\varphi_1$ in the proof of Theorem~\ref{main} by another relation and map, we can improve the upper bound of $g_P(\Gamma)$. 
For example,  for $g=4n$, we consider the following word: 
\begin{align*}
W_1^{2n}=(t_{B_0^\prime}t_{B_1^\prime}\cdots t_{B_{2n}^\prime})^2t_{c_{2n}}^{-1}, 
\end{align*}
where $B_{2k}^\prime$ and $B_{2k-1}^\prime$ are the simple closed curves defined by 
\begin{align*}
&B_{2k}^\prime=a_k b_{k+1}b_{k+2} \cdots b_{2n-k} c_{2n-k}a_{2n+1-k}, \ \ \ 0\leq k\leq n; \\
&B_{2k-1}^\prime=a_k b_kb_{k+1} \cdots b_{2n+1-k} c_{2n+1-k}a_{2n+1-k}, \ \ \ 1\leq k\leq n,
\end{align*}
respectively. 
In \cite{OS2}, it was shown that this word $W_1^{2n}$ is a relator in $\Mod_g^2$. 
Since $W_2^g$ includes two $t_{c_{2n}}$, we can apply twice $\phi$-twisted $W_1^{2n}$-substitution to $W_2^g$, where $\phi$ satisfies $\phi(c_{2n})=c_{2n}$. 
We denote by $W(1,\phi)$ the relator obtained by applying once trivial $W_1^{2n}$-substitution and once $\phi$-twisted $W_1^{2n}$-substitution to $W_2^g$. 
Let 
\begin{align*}
&\varphi^\prime = t_{b_{n+1}}t_{b_{n+2}}t_{b_{n+3}}\cdots t_{b_{2n}} \ \ \ \mathrm{and}\\
&\varphi^\prime_m = t_{a_1}t_{a_2}\cdots t_{a_{n-1}}t_{a_n}^m t_{b_{n+2}}t_{b_{n+3}}\cdots t_{b_{2n}}.
\end{align*}
Then, $\pi_1(X_{W(1,\varphi^\prime)})=\langle a_1,\ldots,a_n\rangle \cong F_n$ for each $n$. 
In \cite{HKM}, it was shown that for each $n$, $\pi_1(X_{W(1,\varphi_m^\prime)})=\langle a_n,b_n \mid a_nb_na_n^{-1}b_n^{-1}, a_n^m\rangle\cong\mathbb{Z}\oplus\mathbb{Z}_m$. 
Therefore, $g_P(F_n)\leq 4n$ and $g_P(\mathbb{Z}\oplus\mathbb{Z}_m)\leq 4$. 
\end{rem}
\begin{rem}
We expect that by twisted substitution techniques, we have analogues of results of Lefschetz fibrations obtained by fiber sum operations for ones with $(-1)$-sections. 
The articles \cite{OS}, \cite{Kor} and \cite{Mo} gave examples of non-holomorphic Lefschetz fibrations by fiber sum operations (and lantern substitutions). 
By twisted substitution techniques (and a lantern substitution), two kinds of non-holomorphic ones with $(-1)$-sections were constructed in \cite{HKM}. 
One is a Lefschetz fibration with non-complex total space, and the other is a Lefschetz fibration violating the ``slope inequality". 
\end{rem}

\section{Construction of a lift of Gurtas' positive relator}\label{Gurtas}
In this section, we prove Proposition~\ref{gurtasx} instead of Proposition~\ref{gurtasrelation}. 
We find that Proposition~\ref{gurtasrelation} is sufficient to prove Proposition~\ref{gurtasx}. 
In order to prove Proposition~\ref{gurtasx}, we prepare Lemma~\ref{korkmaz}, \ref{korkmazx} and Proposition~\ref{hyper4b}.

We give an outline of the proof that $W_2^g$ is a relator in $\Mod_g^2$. 
\begin{lem}\label{korkmaz}
Let $\Sigma_g^2$ be the compact oriented surface of genus $g$ with two boundary components obtained from $\Sigma_g$ by removing two disjoint open disks. 
Let $a_{g+1}$ and $a_{g+1}^\prime$ be the boundary curves of $\Sigma_g^2$ defined by $a_{g+1}$ and $a_{g+1}^\prime=c_ga_{g+1}$, respectively. 
Then, in $\Mod_g^2$, the following relation holds.
\begin{align*}
&t_{a_{g+1}} t_{a^\prime_{g+1}}= 
  \left\{ \begin{array}{ll}
      \displaystyle ( t_{B_0} t_{B_1} t_{B_2} \cdots t_{B_g} t_{c_r} )^2 & \ \ (g=2r) \\[3mm]
      \displaystyle ( t_{B_0} t_{B_1} t_{B_2} \cdots t_{B_g} t_{a_{r+1}}^2 t_{a_{r+1}^\prime}^2 )^2 & \ \ (g=2r+1).
      \end{array} \right.
\end{align*}
\end{lem}
\begin{proof}[Outline of the proof]
We define $\Delta_0=\overline{\Delta}_0=1$. 
Moreover, for each $k=1,\ldots,2g+1$, we define $\Delta_k$ and $\overline{\Delta}_k$ to be the words 
\begin{align*}
&\Delta_k=t_{A_1}t_{A_2}\cdots t_{A_k}& &\mathrm{and}& &\overline{\Delta}_k=t_{A_k}\cdots t_{A_2}t_{A_1},&
\end{align*}
For each $k=0,1,\ldots,g$, $\beta_k$ and $\beta$ are defined by 
\begin{align*}
&\beta_k=\overline{\Delta}_k\Delta_{2g+1-k}\Delta_{2g-k}^{-1}\overline{\Delta}_k^{-1}, &\mathrm{and}& &\beta=\overline{\Delta}_g^{g+1}. 
\end{align*}
Then, by applying the same argument in Section 2 of \cite{Kor} with $\sigma_i$ replaced by $t_{A_i}$, 
we have the following relation 
\begin{align}\label{delta}
&\beta_0\beta_1\beta_2\cdots\beta_g\beta^2=\Delta_{2g+1}\Delta_{2g}\cdots\Delta_3\Delta_2\Delta_1, 
\end{align}
where $\sigma_i$ is the standard generator of the braid group $\mathrm{B}_{2g+2}$ on $2g+2$ strings.

It is easy to check that $\overline{\Delta}_k\Delta_{2g-k}(A_{2g+1-k})=B_k$. 
This gives 
\begin{align*}
t_{B_k}=(\overline{\Delta}_k\Delta_{2g-k})t_{A_{2g+1-k}}(\overline{\Delta}_k\Delta_{2g-k})^{-1}=\overline{\Delta}_k\Delta_{2g+1-k}\Delta_{2g-k}^{-1}\overline{\Delta}_k^{-1}=\beta_k. 
\end{align*}
Therefore, from the relation (\ref{delta}), we have 
\begin{align*}
&t_{B_0}t_{B_1}t_{B_2}\cdots t_{B_g}(\overline{\Delta}_g)^{2g+2}=\Delta_{2g+1}\Delta_{2g}\cdots\Delta_3\Delta_2\Delta_1.
\end{align*}
When $g=2r$ (resp. for $g=2r+1$), by the chain relation $\overline{\Delta}_g^{2g+2}=t_{c_r}$ (resp. $\overline{\Delta}_g^{g+1}=t_{a_{r+1}}t_{a^\prime_{r+1}}$), we have 
\begin{align}\label{garsidehalf}
&\Delta_{2g+1}\Delta_{2g}\cdots\Delta_3\Delta_2\Delta_1=
  \left\{ \begin{array}{ll}
      \displaystyle t_{B_0}t_{B_1}t_{B_2}\cdots t_{B_g}t_{c_r} & \ \ (g=2r) \\[3mm]
      \displaystyle t_{B_0}t_{B_1}t_{B_2}\cdots t_{B_g}t_{a_{r+1}}t_{a_{r+1}^\prime} & \ \ (g=2r+1).
      \end{array} \right.
\end{align}

If we prove that $t_{a_{g+1}}t_{a_{g+1}^\prime} = (\Delta_{2g+1}\Delta_{2g}\cdots\Delta_3\Delta_2\Delta_1)^2$ in $\mathrm{Mod}_g^2$, the assertion follows. 
Note that by the chain relation, we have $\Delta_{2g+1}^{2g+2} = t_{a_{g+1}}t_{a_{g+1}^\prime}$ in $\mathrm{Mod}_g^2$. 
By Lemma 2.1 (a) in \cite{Kor} (i.e. $t_{A_k}\Delta_m=\Delta_m t_{A_{k-1}}$ if $1<k\leq m$), we have 
\begin{align*}
\Delta_{2g+1}^{2g+2}
=&\Delta_{2g+1}\Delta_{2g}t_{A_{2g+1}}\Delta_{2g+1}\Delta_{2g+1}^{2g-1}\\
=&\Delta_{2g+1}\Delta_{2g}\Delta_{2g+1}t_{A_{2g}}\Delta_{2g+1}^{2g-1}\\
=&\Delta_{2g+1}\Delta_{2g}\Delta_{2g-1}(t_{A_{2g}}t_{A_{2g+1}})t_{A_{2g}}\Delta_{2g+1}^{2g-1}\\
=&\Delta_{2g+1}\Delta_{2g}\Delta_{2g-1}\Delta_{2g+1}(t_{A_{2g-1}}t_{A_{2g}})t_{A_{2g-1}}\Delta_{2g+1}^{2g-2}\\
=&\Delta_{2g+1}\Delta_{2g}\Delta_{2g-1}\Delta_{2g-2}(t_{A_{2g-1}}t_{A_{2g}}t_{A_{2g+1}})(t_{A_{2g-1}}t_{A_{2g}})t_{A_{2g-1}}\Delta_{2g+1}^{2g-2}\\
 &\rotatebox{90}{$\cdots$}&\\ 
=&\Delta_{2g+1}\Delta_{2g}\cdots \Delta_1(t_{A_2}t_{A_3}\cdots t_{A_{2g+1}})(t_{A_2}t_{A_3}\cdots t_{A_{2g}})\cdots(t_{A_2}t_{A_3})t_{A_2}\Delta_{2g+1}\\
=&\Delta_{2g+1}\Delta_{2g}\cdots \Delta_1\Delta_{2g+1}\Delta_{2g}\cdots \Delta_1, \\
\end{align*}
and the proof is complete. 
\end{proof}

\begin{figure}[hbt]
 \centering
     \includegraphics[width=12.5cm]{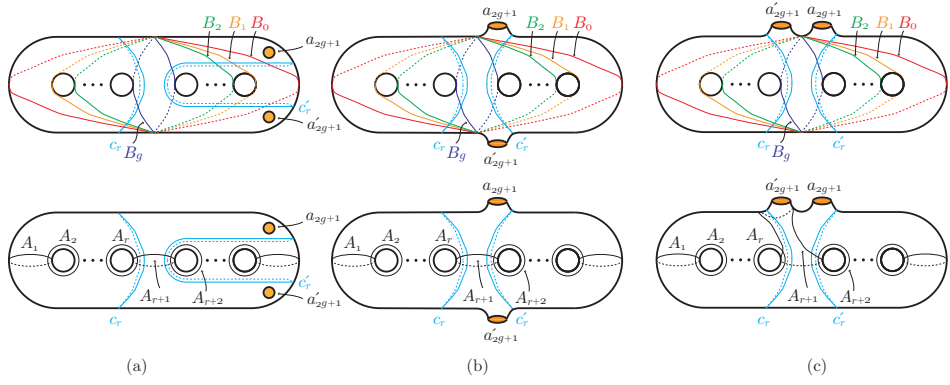}
     \caption{Modified surface $\Sigma_g^2$ and curves $B_0,\ldots,B_g,c_r,c_r^\prime$.}
     \label{crprime}
\end{figure}

\begin{lem}\label{korkmazx}
Suppose that $g=2r$. 
In the notation of Lemma~\ref{korkmaz}, 
let $c_r^\prime$ be the separating simple closed curve defined by $a_{g+1}(b_{r+1}\cdots b_g)a_{g+1}^\prime(b_{r+1}\cdots b_g)^{-1}c_r$ (cf Figure~\ref{crprime} (a)). 
We modify $\Sigma_g^2$ and $B_0,\ldots,B_g$, $c_r, c_r^\prime$ by isotopy as shown in Figure~\ref{crprime} (b) and (c). 
Then, in $\Mod_g^2$, the following relation holds.
\begin{align*}
t_{a_{g+1}}t_{a_{g+1}^\prime}=t_{c_r} t_{c_r^\prime} (t_{B_0}t_{B_1}\cdots t_{B_g})^2.
\end{align*}
\end{lem}
\begin{proof}
It is easily seen that for each $i=1,\ldots, g$, we have 
\begin{align*}
\Delta_{2g+1}\cdots\Delta_2\Delta_1(A_i)=A_{2g+2-i}. 
\end{align*}
This gives the following relation 
\begin{align*}
\Delta_{2g+1}\cdots\Delta_2\Delta_1 t_{A_i}=t_{A_{2g+i}} \Delta_{2g+1}\cdots\Delta_2\Delta_1. 
\end{align*}
for each $i=1,\ldots,2r$. 
Therefore, we have 
\begin{align*}
\Delta_{2g+1}\cdots\Delta_2\Delta_1 (\overline{\Delta}_g)^{-(2g+2)}=(t_{A_{g+2}}\cdots t_{A_{2g+1}})^{-(2g+2)} \Delta_{2g+1}\cdots\Delta_2\Delta_1. 
\end{align*}
Since 
\begin{align*}
&t_{B_0}t_{B_1}t_{B_2}\cdots t_{B_g}(\overline{\Delta}_g)^{2g+2}=\Delta_{2g+1}\cdots\Delta_2\Delta_1 \ (=t_{B_0}t_{B_1}t_{B_2}\cdots t_{B_g}t_{c_r}) 
\end{align*}
from the proof of Lemma~\ref{korkmaz}, we have 
\begin{align*}
(t_{A_{g+2}}\cdots t_{A_{2g+1}})^{2g+2} t_{B_0}t_{B_1}t_{B_2}\cdots t_{B_g} & = \Delta_{2g+1}\cdots\Delta_2\Delta_1\\
&( = t_{B_0}t_{B_1}t_{B_2}\cdots t_{B_g}t_{c_r}). 
\end{align*}
By the chain relation, we obtain $t_{c_r^\prime}=(t_{A_{g+2}}\cdots t_{A_{2g+1}})^{2g+2}$. 
Therefore, from Lemma~\ref{korkmaz}, we obtain 
\begin{align*}
t_{a_{g+1}}t_{a_{g+1}^\prime}= t_{c_r^\prime} t_{B_0}t_{B_1}\cdots t_{B_g} \cdot t_{B_0}t_{B_1}\cdots t_{B_g}t_{c_r}. 
\end{align*}
By conjugate by $t_{c_r}$, we have 
\begin{align*}
t_{a_{g+1}}t_{a_{g+1}^\prime}=t_{c_r} t_{c_r^\prime} (t_{B_0}t_{B_1}\cdots t_{B_g})^2. 
\end{align*}
\end{proof}

The following relation was constructed by Hamada \cite{Ha}.
The proof is based on the argument of \cite{Ta}.

\begin{prop}[\cite{Ha}]\label{hyper4b}
Let $\Sigma_g^4$ be the compact oriented surface of genus $g$ with four boundary components obtained from $\Sigma_g$ by removing four disjoint open disks. 
Let $a_0,a_0^\prime,a_{g+1}$ and $a_{g+1}^\prime$ be the boundary curves of $\Sigma_g^2$, and $a_0$ and $a_{g+1^\prime}$ are defined by $a_0^\prime=c_0a_0$ $a_{g+1}^\prime=c_ga_{g+1}$, respectively. 
Then, the following relation in $\Mod_g^4$ holds:
\begin{align*}
&t_{a_0} t_{a_0^\prime} t_{a_{g+1}} t_{a_{g+1}^\prime} = t_{A_{2g+1}}  \cdots t_{A_2} t_{a_1} t_{a_1^\prime} t_{A_2} \cdots t_{A_{2g+1}} \cdot t_{A_1} \cdots t_{A_{2g}} t_{a_g} t_{a_g^\prime} t_{A_{2g}} \cdots t_{A_1}. 
\end{align*}
\end{prop}
\begin{proof}
The proof is by induction on genus.

Suppose that $g=1$. 
The following relation, called the \textit{four-holed torus relation}, was constructed by Korkmaz and Ozbagci (see \cite{KO}, Section 3.4):
\begin{align*}
&t_{a_0} t_{a_0^\prime} t_{a_2} t_{a_2^\prime} = (t_{A_1} t_{A_3} t_{A_2} t_{a_1} t_{a_1^\prime} t_{A_2})^2. 
\end{align*}
Since $a_0,a_0^\prime,a_2,a_2^\prime$ are disjoint from $A_1$ and $A_1$ is disjoint from $A_3$, by conjugation by $t_{A_1}$, we have 
\begin{align*}
t_{a_0} t_{a_0^\prime} t_{a_2} t_{a_2^\prime} & = t_{A_3} t_{A_2} t_{a_1} t_{a_1^\prime} t_{A_2} \red{t_{A_1}} \cdot \red{t_{A_3}} t_{A_2} t_{a_1} t_{a_1^\prime} t_{A_2} t_{A_1} \\ 
& =  t_{A_3} t_{A_2} t_{a_1} t_{a_1^\prime} t_{A_2} \red{t_{A_3}} \cdot \red{t_{A_1}} t_{A_2} t_{a_1} t_{a_1^\prime} t_{A_2} t_{A_1}.
\end{align*}
Hence, the conclusion of the Proposition holds for $g=1$.

We assume, inductively, that the relation holds in $\Mod_{g-1}^4$. 
Since then $a_0,a_0^\prime,a_g,a_g^\prime$ are disjoint from $A_1,\ldots,A_{2g-1}$, 
we have the following relation in $\Mod_g^4$ by conjugation by $t_{A_{2g-2}} \cdots t_{A_1}$:
\begin{align*}
&t_{a_0} t_{a_0^\prime} t_{a_g} t_{a_g^\prime} \\
&= t_{A_{2g-2}} \cdots t_{A_1} \cdot t_{A_{2g-1}} \cdots t_{A_2} t_{a_1} t_{a_1^\prime} t_{A_2} \cdots t_{A_{2g-1}} \cdot t_{A_1} \cdots t_{A_{2g-2}} t_{a_{g-1}} t_{a_{g-1}}^\prime.
\end{align*}
Since $a_{g-1},a_{g-1}^\prime, a_{g+1}, a_{g+1}^\prime$ are disjoint from $A_{2g-1},A_{2g},A_{2g+1},a_g,a_g^\prime$, 
by the four-holed torus relation 
\begin{align*}
t_{a_{g-1}} t_{a_{g-1}^\prime} t_{a_{g+1}} t_{a_{g+1}^\prime} = (t_{A_{2g-1}} t_{A_{2g+1}} t_{A_{2g}} t_{a_g} t_{a_g^\prime} t_{A_{2g}})^2
\end{align*} 
and conjugation by $t_{A_{2g-1}} t_{A_{2g+1}} t_{A_{2g}}$, 
we have the following relation:
\begin{align*}
t_{a_g}^{-1} t_{a_g^\prime}^{-1} t_{a_{g+1}} t_{a_{g+1}^\prime} =  t_{a_{g-1}^\prime}^{-1} t_{a_{g-1}}^{-1} t_{A_{2g}} t_{A_{2g-1}} t_{A_{2g+1}} t_{A_{2g}} t_{a_g} t_{a_g^\prime} t_{A_{2g}} t_{A_{2g-1}} t_{A_{2g+1}} t_{A_{2g}}. 
\end{align*}
By combining these relations, we have 
\begin{align*}
t_{a_0} t_{a_0^\prime} t_{a_{g+1}} t_{a_{g+1}^\prime} &= t_{A_{2g-2}} \cdots t_{A_1} \cdot t_{A_{2g-1}} \cdots t_{A_2} t_{a_1} t_{a_1^\prime} t_{A_2} \cdots t_{A_{2g-1}} \cdot t_{A_1} \cdots t_{A_{2g-2}} \\
& \ \cdot t_{A_{2g}} t_{A_{2g-1}} t_{A_{2g+1}} t_{A_{2g}} \cdot t_{a_g} t_{a_g^\prime} t_{A_{2g}} t_{A_{2g-1}} t_{A_{2g+1}} t_{A_{2g}}.
\end{align*}
Note that $A_1,\ldots,A_{2g+1}$ are disjoint from $a_0,a_0^\prime,a_{g+1},a_{g+1}^\prime$. 
Moreover, $A_{2g}$ and $A_{2g+1}$ are disjoint from $A_1,\ldots,A_{2g-2}$ and $A_1,\ldots,A_{2g-1}$, respectively. 
Therefore, by conjugation by $t_{A_{2g-2}}\cdots t_{A_1}$ and $t_{A_{2g+1}}t_{A_{2g}}$, we have
\begin{align*}
&t_{a_0} t_{a_0^\prime} t_{a_{g+1}} t_{a_{g+1}^\prime} \\
&= \blue{t_{A_{2g-2}} \cdots t_{A_1}} \cdot t_{A_{2g-1}} \cdots t_{A_2} t_{a_1} t_{a_1^\prime} t_{A_2} \cdots t_{A_{2g-1}} \cdot t_{A_1} \cdots t_{A_{2g-2}} \\
& \ \cdot \red{t_{A_{2g}}} t_{A_{2g-1}} \red{t_{A_{2g+1}}} t_{A_{2g}} \cdot t_{a_g} t_{a_g^\prime} t_{A_{2g}} t_{A_{2g-1}} \blue{t_{A_{2g+1}} t_{A_{2g}}} \\
&= \blue{t_{A_{2g+1}} t_{A_{2g}}} \cdot t_{A_{2g-1}} \cdots t_{A_2} t_{a_1} t_{a_1^\prime} t_{A_2} \cdots t_{A_{2g-1}} \cdot \red{t_{A_{2g}} t_{A_{2g+1}}} \cdot t_{A_1} \cdots t_{A_{2g-2}} \\
& \ \cdot t_{A_{2g-1}} t_{A_{2g}} \cdot t_{a_g} t_{a_g^\prime} t_{A_{2g}} t_{A_{2g-1}} \blue{t_{A_{2g-2}} \cdots t_{A_1}}. 
\end{align*}

This completes the proof of Proposition~\ref{hyper4b}. 
\end{proof}

\begin{figure}[hbt]
 \centering
     \includegraphics[width=8cm]{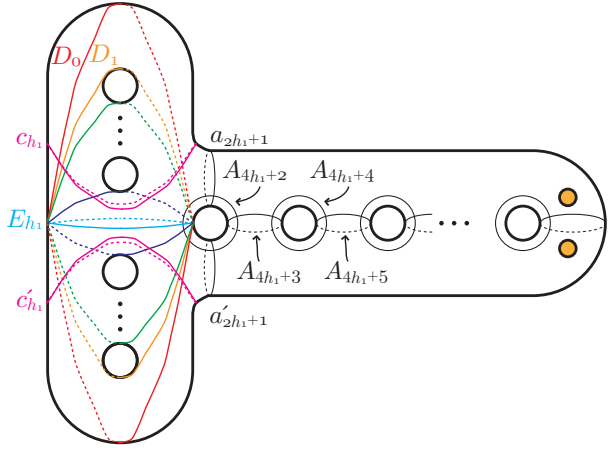}
     \caption{The curve $c_{h_1}^\prime$ on $\Sigma_g^2$.}
     \label{gurtascurve3}
\end{figure}

Proposition~\ref{gurtasrelation} is sufficient to prove the following Proposition. 
\begin{prop}\label{gurtasx}
Let $\Sigma_g^2$ (resp.@$\Sigma_g^1$) be the compact oriented surface of genus $g$ with two boundary components (resp. one boundary component) obtained from $\Sigma_g$ by removing two disjoint open disks (resp. one open disk). 
Let $a_{g+1}, a_{g+1}^\prime=c_ga_{g+1}$ (resp. $a_{g+1}$) be the boundary curves of $\Sigma_g^2$ (resp. the boundary curve of $\Sigma_g^1$). 
Then, the following relations (\ref{gurtas2}) and (\ref{gurtas1}) hold in $\Mod_g^2$ and $\Mod_g^1$, respectively. 
\begin{align}
t_{a_{g+1}} t_{a_{g+1}^\prime} &= t_{E_{h_1}} t_{A_{2h_1+2}} \cdots t_{A_{2g}} t_{a_g} t_{a_g^\prime} t_{A_{2g}} \cdots t_{A_{2h_1+2}} t_{E_{h_1}}  \label{gurtas2} \\
& \ \cdot t_{A_{2g+1}} t_{A_{2g}} \cdots t_{A_{2h_1+2}} \cdot (t_{D_0} t_{D_1} \cdots t_{D_{2h_1}})^2 \cdot t_{A_{2h_1+2}} \cdots t_{A_{2g}} t_{A_{2g+1}}, \notag \\
t_{a_{g+1}}&= t_{E_{h_1}} t_{A_{2h_1+2}} \cdots t_{A_{2g}} t_{a_g} t_{a_g^\prime} t_{A_{2g}} \cdots t_{A_{2h_1+2}} t_{E_{h_1}} \label{gurtas1} \\
& \ \cdot t_{a_g^\prime} t_{A_{2g}} \cdots t_{A_{2h_1+2}} \cdot (t_{D_0} t_{D_1} \cdots t_{D_{2h_1}})^2 \cdot t_{A_{2h_1+2}} \cdots t_{A_{2g}} t_{a_g^\prime}. \notag
\end{align}
\end{prop}
\begin{proof}
Let $c_{h_1}^\prime$ be the separating simple closed curve 
as shown in Figure~\ref{gurtascurve3}. 
By Lemma~\ref{korkmazx} and Proposition~\ref{hyper4b}, we have 
\begin{align*}
&t_{a_{h_1+1}} t_{a_{h_1+1}^\prime} = t_{c_{h_1}} t_{c_{h_1}^\prime} (t_{D_0} t_{D_1} \cdots t_{D_{2h_1}})^2, \\
&t_{c_{h_1}} t_{c_{h_1}^\prime} t_{a_{g+1}} t_{a_{g+1}^\prime} = t_{a_g} t_{A_{2g}} \cdots t_{A_{2h_1+2}} t_{E_{h_1}} t_{E_{h_1}} t_{A_{2h_1+2}} \cdots t_{A_{2g}} t_{a_g^\prime} \\
& \hspace{83pt} \cdot t_{A_{2g+1}} \cdots t_{A_{2h_1+2}} t_{a_{h_1+1}} t_{a_{h_1+1}^\prime} t_{A_{2h_1+2}} \cdots t_{A_{2g+1}}.
\end{align*}
Since $c_{h_1}$ and $c_{h_1}^\prime$ are disjoint from $A_{2h_1+2},\ldots,A_{2g}, E_{h_1},a_{h_1+1},a_{h_1+1}^\prime$, we have 
\begin{align*}
&t_{c_{h_1}^\prime}^{-1} t_{c_{h_1}}^{-1} \cdot t_{a_{h_1+1}} t_{a_{h_1+1}^\prime} = (t_{D_0} t_{D_1} \cdots t_{D_{2h_1}})^2, \\
&t_{a_{g+1}} t_{a_{g+1}^\prime} = t_{a_g} t_{A_{2g}} \cdots t_{A_{2h_1+2}} t_{E_{h_1}} t_{E_{h_1}} t_{A_{2h_1+2}} \cdots t_{A_{2g}} t_{a_g^\prime} \\
& \hspace{50pt} \cdot t_{A_{2g+1}} \cdots t_{A_{2h_1+2}} \cdot t_{c_{h_1}^\prime}^{-1} t_{c_{h_1}}^{-1} \cdot t_{a_{h_1+1}} t_{a_{h_1+1}^\prime} \cdot t_{A_{2h_1+2}} \cdots t_{A_{2g+1}}. 
\end{align*}
Combining these relations gives the relation (\ref{gurtas2}).

In $\Sigma_g^1$, $A_{2g+1}$ is homotopic to $a_g^\prime$, and the relation (\ref{gurtas1}) follows. 
\end{proof}

\begin{figure}[hbt]
 \centering
     \includegraphics[width=11cm]{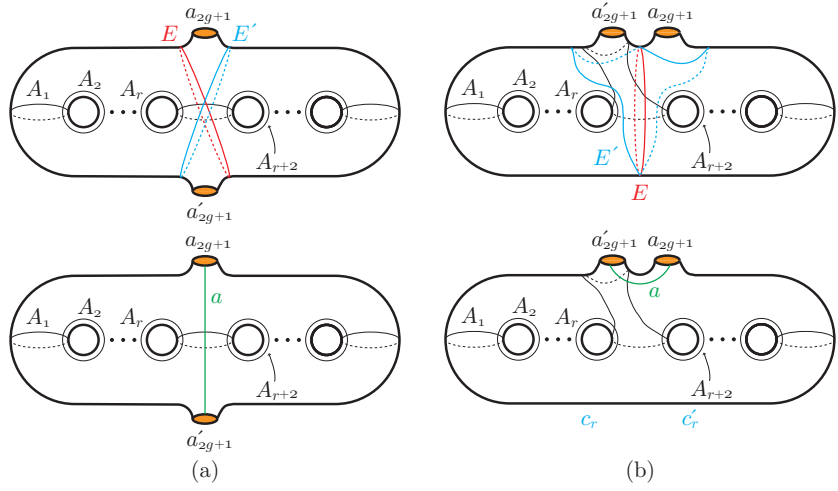}
     \caption{The curves $E,E^\prime$ and the arc $a$.}
     \label{EEprimea}
\end{figure}

Since, in $\Sigma_g$, $a_{g+1}$ and $a_{g+1}^\prime$ are null-homotopic, $t_{a_{g+1}}=t_{a_{g+1}^\prime}=1$ in $\Mod_g$, the relation in Proposition~\ref{gurtasx} is positive relator in $\Mod_g$. 
Then, we note that $A_{2g+1}$ and $a_g^\prime$ are homotopic to $a_g$. 
We prove that, in $\Mod_g$, the relations in Proposition~\ref{gurtasx} are Hurwitz equivalent to Gurtas' positive relator (see \cite{Gur}): 
\begin{align*}
(t_{A_{2h_1+2}} \cdots t_{A_{2g}} t_{a_g} t_{a_g} t_{A_{2g}} \cdots t_{A_{2h_1+2}} t_{D_0} t_{D_1} \cdots t_{D_{2h_1}} t_{E_{h_1}} )^2=1.
\end{align*}
In order to prove this result, we prepare the following Lemma. 
\begin{lem}\label{garside}
We deform $\Sigma_g^2$ as shown in Figure~\ref{EEprimea} (a) and (b). 
Let $E$ and $E^\prime$ be the simple closed curves on $\Sigma_g^2$ as shown in Figure~\ref{EEprimea} (a) and (b), and let $a$ be the arc connecting the boundary components of $\Sigma_g^2$ as shown in Figure~\ref{EEprimea} (a) and (b). 
Then, 
\begin{align}
&t_{B_0}t_{B_1}\cdots t_{B_g}(E)=E^\prime, \label{garsidea}\\
&t_{B_0}t_{B_1}\cdots t_{B_g}t_E(a)=t_{a_{g+1}}t_{a_{g+1}^\prime}(a). \label{garsideb}
\end{align}
\end{lem}
\begin{proof}
From the equality (\ref{garsidehalf}), we see that 
\begin{align*}
t_{B_0}t_{B_1}\cdots t_{B_g}=\Delta_{2g+1}\cdots\Delta_2\Delta_1t_{c_r}^{-1}. 
\end{align*}
By drawing picture pictures, we find that 
\begin{align*}
&\Delta_{2g+1}\cdots\Delta_2\Delta_1t_{c_r}^{-1}(E)=E^\prime, \ \ \ \mathrm{and} \ \ \ \Delta_{2g+1}\cdots\Delta_2\Delta_1t_{c_r}^{-1}t_E(a)=t_{a_{g+1}}t_{a_{g+1}^\prime}(a).
\end{align*}
This proves the Lemma. 
\end{proof}

We prove the following Proposition. 
This means that Proposition~\ref{gurtasx} gives an alternative construction of the monodromy of Gurtas fibration. 
\begin{prop}
In $\Mod_g$, the following relation holds. 
\begin{align*}
& t_{E_{h_1}} t_{A_{2h_1+2}} \cdots t_{A_{2g}} t_{a_g} t_{a_g} t_{A_{2g}} \cdots t_{A_{2h_1+2}} t_{E_{h_1}}  \\
& \ \cdot t_{a_g} t_{A_{2g}} \cdots t_{A_{2h_1+2}} \cdot (t_{D_0} t_{D_1} \cdots t_{D_{2h_1}})^2 \cdot t_{A_{2h_1+2}} \cdots t_{A_{2g}} t_{a_g} \\
& \sim (t_{A_{2h_1+2}} \cdots t_{A_{2g}} t_{a_g} t_{a_g} t_{A_{2g}} \cdots t_{A_{2h_1+2}} t_{D_0} t_{D_1} \cdots t_{D_{2h_1}} t_{E_{h_1}} )^2. 
\end{align*}
\end{prop}
\begin{proof}
For simplity of notation, we write 
\begin{align*}
\tau:=t_{A_{2h_1+2}} \cdots t_{A_{2g}} t_{a_g} \ \ \ \mathrm{and} \ \ \ \overline{\tau}:=t_{a_g} t_{A_{2g}} \cdots t_{A_{2h_1+2}}. 
\end{align*}
Note that for each $i=2h_1+2,\ldots, 2g$, we find that 
\begin{align*}
&t_{E_{h_1}} \tau \overline{\tau} t_{E_{h_1}} (A_i)=A_i \ \ \ \mathrm{and} \ \ \ 
t_{E_{h_1}} \tau \overline{\tau} t_{E_{h_1}} (a_g)=a_g. 
\end{align*}
This gives 
\begin{align*}
& t_{E_{h_1}} \tau \overline{\tau} t_{E_{h_1}} \cdot t_{A_i} \sim t_{A_i} \cdot t_{E_{h_1}} \tau \overline{\tau} t_{E_{h_1}} \ \ \ \mathrm{and} \ \ \ 
 t_{E_{h_1}} \tau \overline{\tau} t_{E_{h_1}} \cdot t_{a_g} \sim t_{a_g} \cdot t_{E_{h_1}} \tau \overline{\tau} t_{E_{h_1}}, 
\end{align*}
so, we obtain the following relation:
\begin{align*}
t_{E_{h_1}} \tau \overline{\tau} t_{E_{h_1}} \cdot \tau \sim \tau \cdot t_{E_{h_1}} \tau \overline{\tau} t_{E_{h_1}}. 
\end{align*}
Therefore, applying elementary transformations (including cyclic permutations) gives 
\begin{align}\label{ele1}
& t_{E_{h_1}} \tau \overline{\tau}t_{E_{h_1}} \cdot \overline{\tau} (t_{D_0} t_{D_1} \cdots t_{D_{2h_1}})^2 \cdot \red{\tau} \sim t_{E_{h_1}} \tau \overline{\tau} t_{E_{h_1}} \cdot \red{\tau} \cdot \overline{\tau} (t_{D_0} t_{D_1} \cdots t_{D_{2h_1}})^2. 
\end{align}

Since by drawing pictures and Lemma~\ref{garside} (\ref{garsideb}), we find 
\begin{align*}
&(\tau \overline{\tau})^{-1} (E_{h_1}) =t_{a_{2h_1+1}}t_{a_{2h_1+1}^\prime}(E_{h_1})
=t_{D_0} t_{D_1} \cdots t_{D_{2h_1}} (E_{h_1}), 
\end{align*}
we obtain 
\begin{align*}
\tau \overline{\tau} \cdot t_{D_0} t_{D_1} \cdots t_{D_{2h_1}} \cdot t_{E_{h_1}} \sim t_{E_{h_1}} \cdot \tau \overline{\tau} \cdot t_{D_0} t_{D_1} \cdots t_{D_{2h_1}}.
\end{align*}
Therefore, by using this relation, we have
\begin{align}
& t_{E_{h_1}} \tau \overline{\tau} \red{t_{E_{h_1}}} \cdot \tau \overline{\tau} \cdot (t_{D_0} t_{D_1} \cdots t_{D_{2h_1}})^2 \label{ele2}\\
& \sim t_{E_{h_1}} \tau \overline{\tau} \cdot \tau \overline{\tau} \cdot t_{D_0} t_{D_1} \cdots t_{D_{2h_1}} \cdot \red{t_{E_{h_1}}} \cdot t_{D_0} t_{D_1} \cdots t_{D_{2h_1}}. \notag
\end{align}

By drawing pictures and Lemma~\ref{garside} (\ref{garsidea}), we find 
\begin{align*}
& (\tau \overline{\tau})^{-1} (A_{2h_1+2})=t_{D_0} t_{D_1} \cdots t_{D_{2h_1}} t_{E_{h_1}} (A_{2h_1+2}). 
\end{align*}
Therefore, we have 
\begin{align*}
& \tau \overline{\tau} \cdot t_{D_0} t_{D_1} \cdots t_{D_{2h_1}} t_{E_{h_1}}  \cdot t_{A_{2h_1+2}} \sim t_{A_{2h_1+2}} \cdot \tau \overline{\tau} \cdot t_{D_0} t_{D_1} \cdots t_{D_{2h_1}} t_{E_{h_1}} .
\end{align*}
Note that for each $i=2h_1+3,\ldots, 2g$, we find that 
\begin{align*}
& \tau \overline{\tau} (A_i)=A_i \ \ \ \mathrm{and} \ \ \ \tau \overline{\tau} (a_g)=a_g. 
\end{align*}
Moreover, since $A_{2h_1+3},\ldots,A_{2g}$ and $a_g$ are disjoint from $D_0,\ldots, D_{2h_1}, E_{h_1}$, 
Therefore, we obtain the following: 
For each $i=2h_1+3,\ldots, 2g$, 
\begin{align*}
& \tau \overline{\tau} t_{D_0} t_{D_1} \cdots t_{D_{2h_1}} t_{E_{h_1}} \cdot t_{A_i} \sim t_{A_i} \cdot  t\tau \overline{\tau} t_{D_0} t_{D_1} \cdots t_{D_{2h_1}} t_{E_{h_1}}, \\
& \tau \overline{\tau} t_{D_0} t_{D_1} \cdots t_{D_{2h_1}} t_{E_{h_1}} \cdot t_{a_g} \sim t_{a_g} \cdot \tau \overline{\tau} t_{D_0} t_{D_1} \cdots t_{D_{2h_1}} t_{E_{h_1}}. 
\end{align*}
This gives 
\begin{align*}
\tau \overline{\tau} \cdot \tau \overline{\tau} t_{D_0} t_{D_1} \cdots t_{D_{2h_1}} t_{E_{h_1}} \sim \tau \overline{\tau} t_{D_0} t_{D_1} \cdots t_{D_{2h_1}} t_{E_{h_1}} \cdot \tau \overline{\tau}. 
\end{align*}
From this relation, applying elementary transformations (including cyclic permutations) gives
\begin{align}
& \blue{t_{E_{h_1}}} \red{\tau \overline{\tau}} \cdot \tau \overline{\tau} \cdot t_{D_0} t_{D_1} \cdots t_{D_{2h_1}} t_{E_{h_1}} \cdot t_{D_0} t_{D_1} \cdots t_{D_{2h_1}} \label{ele3}\\
& \ \sim \tau \overline{\tau} \cdot t_{D_0} t_{D_1} \cdots t_{D_{2h_1}} t_{E_{h_1}} \cdot \red{\tau \overline{\tau}} \cdot t_{D_0} t_{D_1} \cdots t_{D_{2h_1}} \cdot \blue{t_{E_{h_1}}}. \notag
\end{align}

The Proposition follows from the relations (\ref{ele1}),(\ref{ele2}) and (\ref{ele3}). 
This completes the proof. 
\end{proof}

\appendix
\section{}\label{A}
In this Appendix, we prove Proposition~\ref{prop5.1}.

\begin{proof}[Proof of Proposition~\ref{prop5.1}]
Let us consider the surface $\Sigma_n$ embedded in $\mathrm{R}^3$ as shown as in Figure~\ref{fundamental} such that for each $1\leq j\leq n$, a simple closed curves $b_j^\prime$ on $\Sigma_n$ which is isotopic to $b_j$ lies on the plane $x=0$. 
Write $r_i=a_{i_1}^{m_1}\cdots a_{i_d}^{m_d}$, where $d=l(r_i)$ is the syllable length of $r_i$. 
We denote by $\xi$ a constant such that the base point lies in the plane $z=\xi$. 
Let $L$ be a arc on $\Sigma_n$ which lies in the plane $\{z=\xi\}\cap \{x\geq 0\}$.

For $1\leq t\leq d$, let $\alpha_t$ be a loop on $\Sigma_n$ which is isotopic to $a_{i_t}$. 
If $j_s=j_{s^\prime}$ for some $s<s^\prime$, then we assume that $\alpha_{s^\prime}$ is the right of $\alpha_s$ and that $\alpha_{s^\prime}$ is disjoint from $\alpha_s$. 
Here, we call the positive direction of $y$-axis right. 
Let $A_t$ (resp. $B_t$) be a point on $L$ such that $A_t$ (resp. $B_t$) is the left of $\alpha_t$ (resp. the right of $\alpha_t$) and there are not $A_s,B_s$ between $x_t$ and $A_t$ (resp. $B_t$ and $x_t$). 
\begin{figure}[hbt]
 \centering
     \includegraphics[width=10cm]{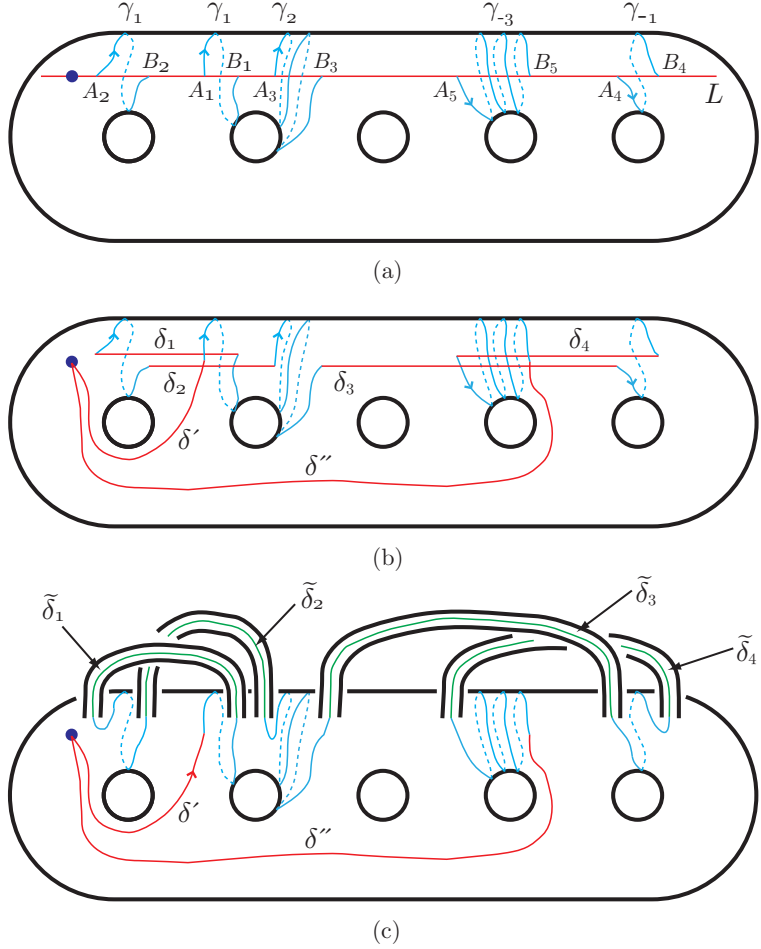}
     \caption{Construction of $R_i$ on $\Sigma_{n+d-1}$ for $r_i=a_2a_1a_2^2a_5^{-1}a_4^{-3}$ and for $n=5$.}
     \label{curveR}
\end{figure}

Let $\gamma_{m_t}=t_{\alpha_t}^{-m_t}(\zeta_t)$, where $\zeta_t$ is the subarc of $L$ from the point $A_j$ to the point $B_j$. 
For each $1\leq j\leq d-1$, let denote by $\delta_j$ the subarc of $L$ from the point $B_j$ to the point $A_{j+1}$. 
Then, we can define an arc $\beta$ on $\Sigma_n$ connecting $A_1$ to $B_d$ to be 
\begin{align*}
\beta=\gamma_{m_1} \star \delta_1 \star \gamma_{m_2} \star \delta_2 \star \cdots \star \delta_{d-1} \star \gamma_{m_d}, 
\end{align*}
where $\gamma \star \delta$ denotes an arc $\gamma$ followed by an arc $\delta$. 
Let $\delta_0$ (resp. $\delta_d$) be the subarc of $L$ from the base point to $A_1$ (resp. from $B_d$ to the base point). 
Then, $\delta_0 \star \beta \star \delta_d$ represents $r_i$ (cf. Figure~\ref{curveR} (a)). 
After perturbing $\beta$ slightly, we assume that $\delta_1,\ldots,\delta_{d-1}$ are pairwise disjoint and lie parallel to the plane $x=0$. 
Note that all self-intersection points of $\delta_0 \star \beta \star \delta_d$ lie on $\delta_0 \cup \delta_1 \cup \cdots \cup \delta_d$.

Let $\delta^\prime$ (resp. $\delta^{\prime\prime}$) be an arc from the base point to $A_1$ (resp. from $B_d$ to the base point) 
which is disjoint from $\alpha_1,\alpha_2,\ldots,\alpha_d$ and $b_1^\prime,b_2^\prime,\ldots,b_n^\prime$ and lies in the space $\{z\leq \xi\}$. 
Suppose that the interiors of $\delta^\prime$, $\delta^{\prime\prime}$ and $\beta$ are pairwise disjoint. 
Then, the loop $\delta^\prime \star \beta \star \delta^{\prime\prime}$ represents 
\begin{align*}
b_1b_2\cdots b_{i_1-1}r_i b_{i_d}^{-1}\cdots b_2^{-1}b_1^{-1}
\end{align*}
in $\pi_1(\Sigma_n)$ (cf. Figure~\ref{curveR} (b)).

Let $D_1,D_1^\prime,\ldots,D_{d-1},D_{d-1}^\prime$ be pairwise disjoint disks on $\Sigma_n$ such that for each $1\leq t\leq d-1$,
$\mathrm{Int}(D_t)$ and $\mathrm{Int}(D_t^\prime)$ are disjoint from $\delta^\prime$, $\beta$ and $\delta^{\prime\prime}$, and $A_t \in \partial D_t$ and $B_t \in \partial D_t^\prime$. 
We remove $2d-2$ open disks $\mathrm{Int}(D_t)$ and $\mathrm{Int}(D_t^\prime)$ from $\Sigma_n$. 
Then, for each $1\leq t\leq d-1$, by attaching an annulus, denote by $\mathcal{A}_t$, to the surface 
\begin{align*}
\displaystyle \Sigma_n \setminus \bigcup_{t=1}^{d-1} (\mathrm{Int}(D_t)\cup \mathrm{Int}(D_t^\prime))
\end{align*} 
along $\partial D_t$ and $\partial D_t^\prime$, we obtain the closed oriented surface 
\begin{align*}
\displaystyle \left(\Sigma_n \setminus \bigcup_{t=1}^{d-1} (\mathrm{Int}(D_t)\cup \mathrm{Int}(D_t^\prime))\right) \cap \left(\bigcup_{t=1}^{d-1}\mathcal{A}_t\right)
\end{align*} 
of genus $n+d-1$, denoted by $\Sigma_{n+d-1}$. 
An orientation on $\Sigma_{n+d-1}$ is given by the orientation on $\Sigma_n$.

We define a loop $R_i$ on $\Sigma_{n+d-1}$ as follows. 
For each $1\leq t \leq d-1$, let $\widetilde{\delta}_t$ be a simple arc on $\mathcal{A}_t$ from the point $B_t$ to the point $A_{t+1}$ 
such that $\widetilde{\delta}_t$ lies parallel to the plane $x=0$. 
Then, by ``replacing" $\delta_t$ in $\delta^\prime \star \beta \star \delta^{\prime\prime}$ by $\widetilde{\delta}_t$, we obtain the loop 
\begin{align*}
R=\delta^\prime \star \gamma_{m_1} \star \widetilde{\delta}_1 \star \gamma_{m_2} \star \widetilde{\delta}_2 \star \cdots \star \widetilde{\delta}_{d-1} \star \gamma_{m_d} \star \delta^{\prime\prime}. 
\end{align*}
In particular, $R_i$ is simple on $\Sigma_{n+d-1}$ (cf. Figure~\ref{curveR} (c)).

Note that from construction, $\widetilde{\delta}_t \star \delta_t$ is a simple closed curve on $\Sigma_{n+d-1}$. 
If we collapse each $\mathcal{A}_t$ onto the arc $\delta_t$, then we obtain a map $\Sigma_{n+d-1}\to \Sigma_n$ and the induced map $\pi_1(\Sigma_{n+d-1}) \to \pi_1(\Sigma_n)$ takes $[R]$ to $b_1b_2\cdots b_{i_1-1}r_i b_{i_d}^{-1}\cdots b_2^{-1}b_1^{-1}$, 
and $b_1b_2\cdots b_{i_1-1}r_i b_{i_d}^{-1}\cdots b_2^{-1}b_1^{-1}$ is mapped to $r_i$ under the map $\pi_1(\Sigma_n)\to\pi_1(\Sigma_n)$ which maps $a_j$ to $a_j$, and $b_j$ to $1$ for all $j$.

Let $h=n+l-1$, where $l=\max_{1\leq i\leq k}\{l(r_i)\}$. 
For each $1\leq i \leq k$, we now construct a loop $R_i$ on $\Sigma_h$ as follows. 
First, by sliding $\mathcal{A}_1,\ldots,\mathcal{A}_{l(r_i)-1}$, we deform the surface $\Sigma_{n+l(r_i)-1}$ into the standard position as shown in Figure~\ref{fundamental} 
in such a way that the simple loop $\widetilde{\delta}_t \star \delta_t$ becomes isotopic to $b_{n+t}$ 
and the boundary curves of $\mathcal{A}_t$ become isotopic to $a_{n+t}$ (cf. Figure~\ref{embedded} (a), (b) and (c)). 
If $l(r_j)=l$ for some $j$, then we see that the simple closed curve $a_h$ intersects $R_j$ transversely at one point. 
Therefore, we assume that $l(r_i)<l$. 
Next, we remove a small open disk from the deformed surface near $a_{n+l(r_i)-1}$ and disjoint from $R_i$ (cf. Figure~\ref{embedded} (d)). 
Thus, we obtain a surface of genus $n+l(r_i)-1$ with one boundary component, denoted by $\Sigma_{n+l(r_i)-1}^1$. 
We embed $\Sigma_{n+l(r_i)-1}^1$ into the standard surface $\Sigma_h$ in such a way that 
for each $1\leq t\leq n+l(r_i)-1$, simple loops $a_t,b_t$ on $\Sigma_{n+l(r_i)-1}^1$ correspond to the simple loops $a_t,b_t$ on $\Sigma_h$ (cf. Figure~\ref{embedded} (e)). 
Finally, we replace $R_i$ with a simple representative of $[R_i]\{(b_1b_2\cdots b_{h-1})(b_1b_2\cdots b_h)^{-1}\}^{\epsilon}$, where $\epsilon=\pm 1$ (cf. Figure~\ref{embedded} (d)). 
Then, we see that the resulting simple loop $R_i$ intersects $a_h$ transversely at one point.

From the above construction, $[R_i]$ is mapped to $r_i$ under the map $\Phi:\pi_1(\Sigma_h)\to \pi_1(\Sigma_n)$ for each $i=1,\ldots,k$.

This finish the proof of Proposition~\ref{prop5.1}.
\end{proof}
\begin{figure}[hbt]
 \centering
     \includegraphics[width=10cm]{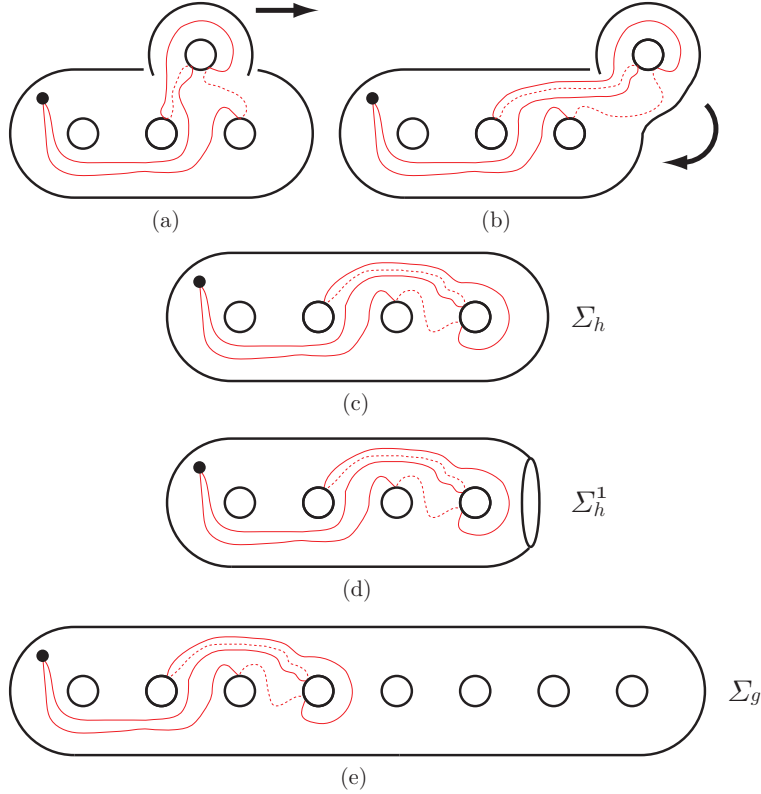}
     \caption{Construction of $R_i$ for $r_i=a_3^{-1}a_2^{-1}$ in the case $n=3$ and $g=8$.}
     \label{embedded}
\end{figure}

\end{document}